\documentclass[11pt,oneside,leqno]{amsart}
\usepackage{amsxtra}
\usepackage{amsopn}
\usepackage{color}
\usepackage{amsmath,amsthm,amssymb}
\usepackage{amscd}
\usepackage{amsfonts}
\usepackage{latexsym}
\usepackage{verbatim}
\usepackage{multirow}
\usepackage{lscape}
\usepackage{tabularx}
\usepackage{multirow}
\usepackage{enumerate}

\theoremstyle{plain}
\newtheorem{theorem}{Theorem}[section]

\newtheorem{lemma}[theorem]{Lemma}
\newtheorem{proposition}[theorem]{Proposition}
\newtheorem{corollary}[theorem]{Corollary}
\newtheorem{remark}[theorem]{Remark}

\newtheorem{example}[theorem]{Example}

\newtheorem{remark-question}[section]{Remark-Question}

\newcommand\B{{\mathbb B}}
\newcommand\RR{{\mathbb R}}
\newcommand\C{{\mathbb C}}

\newcommand\GL{{\hbox{\rm GL}}}

\newcommand\R{{\mathbb R}}

\newcommand\treJF{{\Theta^{\varepsilon,\rho}_{\!(\!J,F)}}}
\newcommand\ureJF{{\Upsilon^{\varepsilon,\rho}_{\!(\!J,F)}}}
\newcommand\nre{{\nabla^{\varepsilon,\rho}}}
\newcommand\nreJF{{\nabla_{\!\!(\!J,F)}^{\varepsilon,\rho}}}

\newcommand\re{{\varepsilon,\rho}}
\newcommand\frb{{\mathfrak b}}
\newcommand\fre{{\mathfrak e}}
\newcommand\frg{{\mathfrak g}}
\newcommand\frh{{\mathfrak h}}

\newcommand\Real{{\mathfrak R}{\frak e}\,} 
\newcommand\Imag{{\mathfrak I}{\frak m}\,}

\newcommand\db{{\bar{\partial}}}
\newcommand\zzz{{\!\!\!}}

\begin{document}
\title[]{Six dimensional homogeneous spaces with holomorphically trivial canonical bundle}
\subjclass[2000]{
}

\author{Antonio Otal}
\address[A. Otal]{Centro Universitario de la Defensa\,-\,I.U.M.A.\\
Academia General Militar\\
Ctra. de Huesca s/n. 50090 Zaragoza\\
Spain}
\email{aotal@unizar.es}

\author{Luis Ugarte}
\address[L. Ugarte]{Departamento de Matem\'aticas\,-\,I.U.M.A.\\
Universidad de Zaragoza\\
Campus Plaza San Francisco\\
50009 Zaragoza, Spain}
\email{ugarte@unizar.es}

\maketitle

\begin{abstract}
We classify all the $6$-dimensional unimodular Lie algebras $\frg$ admitting a complex structure with non-zero closed $(3,0)$-form.
This gives rise to
$6$-dimensional compact homogeneous spaces $M=\Gamma\backslash G$, where $\Gamma$ is a lattice, admitting an
invariant complex structure with holomorphically trivial canonical bundle.
As an application, in the balanced Hermitian case, we study the instanton condition for any metric connection $\nre$ in the plane generated by the Levi-Civita connection and the Gauduchon line of Hermitian connections.
In the setting of the Hull-Strominger system with connection on the
tangent bundle being Hermitian-Yang-Mills, we prove that if a compact non-K\"ahler homogeneous space $M=\Gamma\backslash G$
admits an invariant solution
with respect to some non-flat connection $\nabla$ in the family $\nre$,
then
$M$ is a nilmanifold with underlying Lie algebra $\frh_3$, a solvmanifold with underlying algebra $\frg_7$, or a quotient of the semisimple group SL(2,$\C$).
Since it is known that the system can be solved on these spaces, our result implies that
they are the unique compact non-Kähler balanced homogeneous spaces
admitting such invariant solutions.
As another application, on the compact solvmanifold underlying the
Nakamura manifold,
we construct solutions, on any given balanced Bott-Chern class, to the heterotic equations of motion taking the Chern connection as (flat) instanton.
\end{abstract}



\section{Introduction}

\noindent
Complex manifolds with holomorphically trivial canonical bundle, possibly endowed with a special Hermitian metric, play a relevant role both in geometry and in  theoretical physics. An important source of these distinguished manifolds is provided by certain quotients of Lie groups $G$ by cocompact lattices $\Gamma$, more specifically, by $2n$-dimensional compact homogeneous spaces $M=\Gamma\backslash G$ endowed with an \emph{invariant} complex structure with holomorphically trivial canonical bundle.
By the latter we mean that the (unimodular) Lie algebra $\frg$ of $G$ has a complex structure with non-zero closed $(n,0)$-form.
For instance, when $M$ is a nilmanifold, i.e. $G$ is nilpotent, a result of Salamon \cite{S} ensures that any invariant complex structure on $M$ possesses a non-zero $(n,0)$-form which is closed.

In this paper we are interested in the complex dimension three, mainly due to its relation to the Hull-Strominger system. The precise definition of the system is given below and at this point we just recall that its solutions require in particular a compact complex manifold $X=(M,J)$ with non-vanishing holomorphic $(3,0)$-form $\Psi$.
When $M=\Gamma\backslash G$ is a $6$-dimensional nilmanifold
and $J$ is invariant,
the problem of determining which nilpotent Lie algebras $\frg$ admit a complex structure was completely solved in~\cite{S}.
Furthermore, the unimodular solvable Lie algebras $\frg$ of dimension $6$ that admit a complex structure with non-zero closed $(3,0)$-form are classified in \cite{FOU},
together with the existence problem for several special classes of Hermitian metrics.
Our first goal in this paper is to complete these previous results
by extending the classification to any unimodular Lie algebra in six dimensions.

In Section~\ref{section-2} we study the existence of complex structures with non-zero closed $(3,0)$-form on Lie algebras $\frg$ of real dimension $6$ which are unimodular and non-solvable. Different approaches are followed depending on the decomposability of the algebra $\frg$, but in any case our study relies on the analysis of stable forms in six dimensions~\cite{Hitchin}.
The main result is an uniqueness theorem in this setting, namely that $\mathfrak{so}(3,1)$, i.e. the real Lie algebra underlying $\mathfrak{sl}(2,\C)$, is the only Lie algebra admitting this kind of complex structures (see Theorem~\ref{clasif-non-solvable}).
For completeness, in the Corollaries~\ref{classification}, \ref{balanced-classification}~and~\ref{quotient-balanced} we collect this new result together with other known results to provide classifications of $6$-dimensional unimodular Lie algebras admitting complex structures $J$ with non-zero closed $(3,0)$-form, as well as of those having balanced Hermitian metrics.

\smallskip

As a first application of the above classification, we study the existence of invariant solutions of the Hull-Strominger system with connection $\nabla$ on the
tangent bundle being Hermitian-Yang-Mills. Such solutions satisfy the heterotic equations of motion.
In Section~\ref{HS-homogeneous} we review the definition and several important results about the system found in  \cite{A-Gar,CHSW,COMS,OS,FeiYau,FIUV09,FGV,Fu-Yau,MGarcia-Crelle,GFGM,GP,Hull,Iv,Li-Yau,OUV17,Str}. In particular, in \cite{FIUV09} invariant solutions of the heterotic equations of motion were first obtained on a nilmanifold with underlying Lie algebra $\frh_3$, whereas in \cite{OUV17} new solutions were found on a solvmanifold with underlying algebra $\frg_7$ and on the quotient of the semisimple group with Lie algebra $\mathfrak{so}(3,\!1)$.
In all these solutions, $\nabla$ is taken as the Strominger-Bismut connection~\cite{Bismut}, which is a non-flat instanton. Moreover, it was conjectured in \cite[Section 7]{OUV17} that these are the only spaces admitting such solutions; more concretely, if a compact non-K\"ahler homogeneous space $M=\Gamma\backslash G$
admits an invariant solution of the heterotic equations of motion with slope parameter $\alpha'>0$ and
with respect to some non-flat connection $\nabla$ in the ansatz $\nre$,
then $\nabla$ is the Bismut connection and $M$ is one of the three spaces above.

The connections $\nre$ constitute a plane of metric connections
where
important connections proposed for the anomaly cancellation equation live:
the Levi-Civita connection $\nabla^{LC}=\nabla^{0,0}$, the Hull connection $\nabla^-=\nabla^{-\frac12,0}$, the Chern connection $\nabla^c=\nabla^{0,\frac12}$ and
the Strominger-Bismut connection $\nabla^+=\nabla^{\frac12,0}$,
so also the Gauduchon line $\nabla^{\tau}$ of Hermitian connections~\cite{Gau} joining
$\nabla^+$ and $\nabla^c$ (see Section~\ref{HS-homogeneous} for details).
In Sections~\ref{instantones-nil} and~\ref{instantones-solv} we prove the following result related to this conjecture, which is valid independently of the sign of the slope parameter $\alpha'$:
let $M=\Gamma\backslash G$ be a $6$-dimensional compact manifold defined as the quotient of a simply connected Lie group $G$ by a lattice $\Gamma$ of maximal rank, and suppose that $M$ possesses an invariant complex structure $J$ with non-zero closed $(3,0)$-form admitting an invariant balanced metric $F$.
If some $\nre$ in the associated $(\re)$-plane of metric connections is a non-flat instanton, then the Lie algebra of $G$ is isomorphic to $\frh_3$, $\frg_7$, or $\mathfrak{so}(3,\!1)$
(see Theorem~\ref{main-instanton}).
For the proof, we study the nilpotent case in Section~\ref{instantones-nil}, whereas Section~\ref{instantones-solv} is devoted to the class of solvmanifolds.

\smallskip

As a second application, we consider the Chern connection to construct solutions on the Nakamura manifold with given balanced class.
In greater detail, in Section~\ref{nuevas-soluciones} we take $X$ as the compact complex manifold defined by the Nakamura manifold endowed with its abelian complex structure, and provide solutions of the heterotic equations of motion in the Bott-Chern class of any given invariant balanced metric $F$.
More concretely, in Theorem~\ref{new-solutions} it is proved that given any such $F$, there always exists another balanced metric $\tilde{F}$ with $[\tilde{F}^2]=[F^2]\in H_{\rm BC}^{2,2}(X,\R)$ and a non-flat instanton solving the heterotic equations of motion with respect to the (flat) Chern connection associated to $\tilde{F}$.

\section{Complex structures with closed $(3,0)$-form on non-solvable spaces}\label{section-2}

\noindent In this section we study the existence of complex structures with non-zero closed $(3,0)$-form on Lie algebras $\frg$ of real dimension six. Since the nilpotent and the solvable unimodular cases have been studied respectively in \cite{S} and \cite{FOU}, we will focus on the non-solvable unimodular setting. The conclusion will be that only $\mathfrak{so}(3,1)$ admits this special type of complex structures.
For completeness, in Section~\ref{resultados-classif} we collect this new result together with other known results to provide classifications for Lie algebras (see Corollary~\ref{classification}) and for compact homogeneous spaces with balanced Hermitian metrics (see Corollary~\ref{quotient-balanced}).

\smallskip

We will use different approaches depending on the decomposability of the algebra $\frg$, but in any case our study relies on the analysis of stable forms in six dimensions.

\subsection{Stable forms in dimension six}

Stable forms on $6$-dimensional vector spaces were widely studied in~\cite{Hitchin}. We recall the basic properties
omitting details which can be found in the referred paper.

Let $(V,\nu)$ be an  oriented six-dimensional real vector space, being $\nu\in\wedge^6 V^*$ a fixed volume form of $V$.
A three-form
$\rho\in\wedge^3 V^*$ is stable if its orbit under the action of $\GL(V)$ is open.

Stability can be characterized algebraically in the following way. Let $\kappa\colon\wedge^5V^*\to V\otimes\wedge^6V^*$ be the canonical isomorphism $\kappa(\xi)=v\otimes\nu$, with $\iota_v\nu=\xi$, and define
\begin{equation*}
  \begin{array}{l}
    K_\rho(v):=\kappa(\iota_v\rho\wedge\rho)\in V\otimes\wedge^6V^*,\qquad v\in V,\\[1em]
    \lambda(\rho):=\frac16\text{tr}(K_\rho^2)\in(\wedge^6V^*)^{\otimes 2}.
  \end{array}
\end{equation*}
Then, $\rho$ is stable if and only if $\lambda(\rho)\neq0$. A stable three-form $\rho$ defines a specific volume form
\begin{equation*}
  \phi(\rho)=\sqrt{|\lambda(\rho)|}\in\wedge^6V^*
\end{equation*}
and an endomorphism given by $J_\rho:=\frac{1}{\phi(\rho)}K_\rho$.

It turns out that $J_\rho$ is an almost complex structure if and only if $\lambda(\rho)<0$. In this case we say that $J_\rho$ is
the almost complex structure induced by $\rho$.
The dual almost complex structure, which along this paper we will denote again by $J_\rho$  (instead of $J_\rho^*$), acts on one-forms by the following formula
\begin{equation}\label{compleja_uno_formas}
  (J_\rho\alpha)(v)\,\phi(\rho)=\alpha\wedge\iota_v\rho\wedge\rho,\qquad \alpha\in V^*,\quad v\in V.
\end{equation}
In addition, the complex three-form $\Psi=\rho+iJ_\rho\rho$ has bidegree $(3,0)$ with respect to $J_\rho$.

\smallskip

Now, let $\frg$ be a $6$-dimensional Lie algebra and $J$ an almost complex structure on $\frg$, i.e. $J$ is an endomorphism $J\colon \frg\longrightarrow\frg$ satisfying $J^2=-{\rm Id}_\frg$. The almost complex structure $J$ is called integrable if it has no torsion, i.e. its Nijenhuis tensor
\begin{equation}\label{nijenhuis_tensor}
N_J(X,Y)=[JX,JY]-J[X,JY]-J[JX,Y]-[X,Y],\qquad X,\,Y\in\frg,
\end{equation}
vanishes identically. We are concerned with complex structures satisfying the stronger condition given by the existence of a non-zero closed $(3,0)$-form. As all of them arise from closed stable three-forms $\rho\in\wedge^3 \frg^*$, we will proceed as follows.

Let $\rho$ be a generic closed three-form on the Lie algebra $\frg$. Recall that the differential $d$ is induced from the formula
\begin{equation}\label{formula-d}
  d\alpha(X,Y)=-\alpha([X,Y]),\qquad \alpha\in\frg^*, \quad X,\, Y\in\frg.
\end{equation}
Now, consider the endomorphism $\tilde J_\rho$ defined by acting on one-forms as follows
\begin{equation}\label{criterio}
  \left( (\tilde J_\rho\alpha)(X) \right) \nu=\alpha\wedge\iota_X\rho\wedge\rho,\qquad\alpha\in\frg^*,\quad X\in\frg,
\end{equation}
where $\nu\in\wedge^6\frg^*$ is a
fixed volume form of $\frg$. From~\eqref{compleja_uno_formas} and~\eqref{criterio} we have that the endomorphisms $\tilde J_\rho$ and
the dual of $K_\rho$ coincide.
Then we define $\tilde \lambda(\rho)$ as the scalar given by
$$
\tilde \lambda(\rho) \, \nu^{\otimes 2}=\lambda(\rho)=\frac16\text{tr}(K_\rho^{\,2})=\frac16\text{tr}(\tilde J_\rho^{\,2}).
$$
When $\tilde \lambda(\rho)<0$ and $d(\tilde J_\rho\rho)=0$ we get an almost complex structure $J_\rho$ on $\frg$ with non-zero closed $(3,0)$-form $\Psi=\rho+i\, J_\rho\rho$.
It is straightforward
that the condition $d\Psi=0$ implies that the differential of any (1,0)-form has vanishing (0,2)-component, which in turn is equivalent to the Nijenhuis tensor \eqref{nijenhuis_tensor} of $J_\rho$ being identically zero, i.e. $J_\rho$ is integrable.

As a consequence, in order to prove that a given Lie algebra $\frg$ does not admit any complex structure with closed $(3,0)$-form, it is enough to show
that, for
any closed $\rho\in\wedge^3 \frg^*$,  one gets ${\rm tr}(\tilde J_\rho^{\,2})\geq0$ whenever
$\tilde J_\rho\rho$ is closed.
However, this condition becomes very intricate to deal with in the case
of non-solvable $3\oplus3$ decomposable Lie algebras, and we will use instead the fact that the almost complex structures $J_{\rho}$ satisfying
$N_{J_\rho}=0$ would induce a family of linear endomorphisms on the first $3$-dimensional factor which are not torsion free, so contradicting the existence of such a $J_\rho$ (see Section~\ref{Sec_decomposable} and the proof of Proposition~\ref{prop_sl2R} for the precise argument).

\smallskip

As it is explained in the introduction, we are interested in the geometry of compact complex manifolds of the form $M=\Gamma\backslash G$, where $\Gamma$ is a lattice, endowed with an
invariant complex structure with holomorphically trivial canonical bundle.
Hence, if $\frg$ is the Lie algebra of $G$, in addition to have a complex structure with non-zero closed $(3,0)$-form, $\frg$ must be unimodular due to  a well-known result by Milnor \cite{Milnor} (see also \cite{FOU} for other necessary conditions on the cohomology of $\frg$).

\subsection{Six-dimensional unimodular non-solvable Lie algebras}

In this section we recall the different classes of unimodular non-solvable Lie algebras in six dimensions.
When the Lie algebra is decomposable, i.e. $\frg=\bigoplus_{j=1}^n\frg_j$, the unimodularity of $\frg$ requires
the unimodularity of every summand $\frg_j$. Taking this fact into account, together with the lists of non-solvable Lie algebras up to dimension 5 (see~\cite[Table 1]{Schu} and~\cite[Table 2]{F-Schu1}),
one has the following classification of decomposable unimodular non-solvable Lie algebras of dimension six:
\begin{equation}\label{lista-descomp}
\begin{array}{l}
\mathfrak{sl}(2,\RR)\oplus\RR^3 = (e^{23},\,-e^{13},\,-e^{12},\,0,\,0,\,0),\\[3pt]
\mathfrak{sl}(2,\RR)\oplus\frh_3 = (e^{23},\,-e^{13},\,-e^{12},\,0,\,0,\,e^{45}),\\[3pt]
\mathfrak{sl}(2,\RR)\oplus\fre(1,1) = (e^{23},\,-e^{13},\,-e^{12},\,0,\,-e^{46},\,-e^{45}),\\[3pt]
\mathfrak{sl}(2,\RR)\oplus\fre(2) = (e^{23},\,-e^{13},\,-e^{12},\,0,\,-e^{46},\,e^{45}),\\[3pt]
\mathfrak{sl}(2,\RR)\oplus\mathfrak{sl}(2,\RR) = (e^{23},\,-e^{13},\,-e^{12},\,e^{56},\,-e^{46},\,-e^{45}),\\[3pt]
\mathfrak{sl}(2,\RR)\oplus\mathfrak{so}(3) = (e^{23},\,-e^{13},\,-e^{12},\,e^{56},\,-e^{46},\,e^{45}),\\[3pt]
\mathfrak{so}(3)\oplus\RR^3 = (e^{23},\,-e^{13},\,e^{12},\,0,\,0,\,0),\\[3pt]
\mathfrak{so}(3)\oplus\frh_3 = (e^{23},\,-e^{13},\,e^{12},\,0,\,0,\,e^{45}),\\[3pt]
\mathfrak{so}(3)\oplus\fre(1,1) = (e^{23},\,-e^{13},\,e^{12},\,0,\,-e^{46},\,-e^{45}),\\[3pt]
\mathfrak{so}(3)\oplus\fre(2) = (e^{23},\,-e^{13},\,e^{12},\,0,\,-e^{46},\,e^{45}),\\[3pt]
\mathfrak{so}(3)\oplus\mathfrak{so}(3) = (e^{23},\,-e^{13},\,e^{12},\,e^{56},\,-e^{46},\,e^{45}),\\[3pt]
A_{5,40}\oplus\RR = (2e^{12},\,-e^{13},\,2e^{23},\,e^{24}+e^{35},\,e^{14}-e^{25},\,0).
\end{array}
\end{equation}

All the Lie algebras in \eqref{lista-descomp} are $3\oplus3$ decomposable except $A_{5,40}\oplus\RR$, which is $5\oplus1$ decomposable.
For the description of the structure of each Lie algebra we are using the exterior derivative $d$ instead of the Lie bracket due to the formula
\eqref{formula-d}. In greater detail, for instance the notation $\mathfrak{sl}(2,\RR)\oplus\RR^3=(e^{23},\,-e^{13},\,-e^{12},\,0,\,0,\,0)$ means that the Lie algebra has a basis of one-forms $\{e^j\}_{j=1}^{6}$ such
that
\begin{equation*}
de^1=e^2\wedge e^3,\ \ de^2=-e^1\wedge e^3,\ \ de^3=-e^1\wedge e^2,\ \  de^4=de^5=de^6=0.
\end{equation*}

\smallskip

In the indecomposable case, we are taking the classification from~\cite[Table 2]{F-Schu2}, so up to isomorphism one has the following
list of indecomposable unimodular non-solvable Lie algebras in six dimensions:
\begin{equation}\label{lista-indescomp}
\begin{array}{l}
L_{6,1} = (e^{23},\, -e^{13},\, e^{12},\, e^{26} - e^{35},\, -e^{16} + e^{34},\, e^{15} - e^{24}),\\[3pt]
L_{6,2} = (e^{23},\, 2e^{12},\, -2e^{13},\, e^{14} + e^{25},\, -e^{15} + e^{34},\, e^{45} ),\\[3pt]
L_{6,4} = (e^{23},\, 2e^{12},\, -2e^{13},\, 2e^{14} + 2e^{25},\, e^{26} + e^{34},\, -2e^{16} + 2e^{35} ),\\[3pt]
\mathfrak{so}(3,1) = (e^{23}-e^{56},\,-e^{13}+e^{46},\,e^{12}-e^{45},\,e^{26}-e^{35},\,-e^{16}+e^{34},\,e^{15}-e^{24}).
\end{array}
\end{equation}

If we remove the unimodularity condition, then another Lie algebra, labeled as  $L_{6,3}$, appears. It turns out that this algebra has complex structures with non-zero closed $(3,0)$-form (see Remark~\ref{L63} for details), however no compact quotient of the corresponding simply-connected Lie group by a lattice exists.

It is well-known that $\mathfrak{so}(3,1)$
admits complex structures with closed $(3,0)$-form:

\begin{example}\label{so(3,1)}
The real Lie algebra $\mathfrak{so}(3,1)$ underlies the  $3$-dimensional complex Lie algebra $\mathfrak{sl}(2,\C)$ given by the complex structure equations
\begin{equation}\label{complex-para-str}
d\omega^1=\omega^{23},\quad d\omega^2=-\omega^{13},\quad d\omega^3=\omega^{12}.
\end{equation}
To see it, it suffices to check that the complex one-forms $\{\omega^k\}_{k=1}^{3}$ are forms of bidegree $(1,0)$ with respect to an almost complex structure on
$\mathfrak{so}(3,1)$. Let $J$ be the almost complex structure on $\mathfrak{so}(3,1)$ defined, in terms of the real basis $\{e^j\}_{j=1}^{6}$ given in \eqref{lista-indescomp}, by
\begin{equation*}
Je^1=e^4,\quad Je^2=e^5,\quad Je^3=e^6,\quad Je^4=-e^1,\quad Je^5=-e^2,\quad Je^6=-e^3.
\end{equation*}
Now, consider the $(1,0)$-forms with respect to $J$ given by $\omega^1=e^3-ie^6$, $\omega^2=e^1-ie^4$, and $\omega^3=e^2-ie^5$. A direct calculation
shows that $\omega^1,\,\omega^2,\, \omega^3$ satisfy~\eqref{complex-para-str} and, as $J$ is complex parallelizable, the $(3,0)$-form $\Psi=\omega^{123}$ is closed.
Note that $J$ corresponds to the closed stable three-form $\rho=e^{123}-e^{156}+e^{246}-e^{345}\in \wedge^3\mathfrak{so}(3,1)^*$.
\end{example}

\subsection{The decomposable case}\label{Sec_decomposable}

When the Lie algebra is decomposable we will use the observation in~\cite[Lemma 1]{Magnin} that the integrability of $J$ induces a torsion free endomorphism on every summand.

By definition, a \emph{torsion free endomorphism} on a Lie algebra $\frh$ is a
vector space homomorphism $F\colon\frh\to\frh$ satisfying $N_{F}(X,Y)= 0$ for any $X,Y\in\frh$, where
\begin{equation}\label{torsion_tensor}
N_{F}(X,Y)=[FX,FY]-F[X,FY]-F[FX,Y]-[X,Y].
\end{equation}
Note that the identically zero endomorphism $F\equiv 0$ has zero torsion if and only if the Lie algebra $\frh$ is abelian, and that there are Lie algebras not admitting any torsion free endomorphism.

Now, let $\frg=\bigoplus_{j=1}^n\frg_j$ be a decomposable Lie algebra, and denote by $i_j\colon\frg_j\to\frg$ and $\pi_j\colon\frg\to\frg_j$ the natural inclusion and projection, respectively, for every $j$-th summand. Let $J\colon\frg\to\frg$ be an almost complex structure on $\frg$ and define, for every $j$ with $1\leq j\leq n$,
the endomorphism $F_j=\pi_j\circ J\circ i_j\colon \frg_j\to\frg_j$.

Suppose that $J$ is integrable, i.e. the Nijenhuis tensor $N_J\equiv 0$. Taking the  projection on every summand $\frg_j$, it follows from \eqref{nijenhuis_tensor} that
$$
[\pi_j JX,\pi_j JY]_{\frg_j}-\pi_j J([X,\pi_j JY]_{\frg_j})-\pi_j J([\pi_j JX,Y]_{\frg_j})-[X,Y]_{\frg_j}=0
$$
for any $X,Y\in\frg_j$. Therefore, the endomorphism $F_j\colon \frg_j\to\frg_j$
satisfies $N_{F_j}\equiv 0$, so we have:

\begin{lemma}\label{lema_Magnin}\cite{Magnin}
If $J\colon\frg\to\frg$ is integrable, then $F_j\colon \frg_j\to\frg_j$ is a torsion free endomorphism for every $j$.
\end{lemma}

This will be particularly useful in the case of  $3\oplus3$ decomposable Lie algebras.
As a first step we have the following lemma stating that certain endomorphisms of the Lie algebras $\mathfrak{sl}(2,\R)$ and $\mathfrak{so}(3)$ are not torsion free.

\begin{lemma}\label{lema_sl2R}
For any $A,\,B,\,C,\,\lambda,\,\mu,\,\tau\in\mathbb{R}$, we have:
\begin{enumerate}
  \item[{\rm(i)}] The endomorphism $F\colon\mathfrak{sl}(2,\R)\to\mathfrak{sl}(2,\R)$ defined by the coordinate matrix
\begin{equation}\label{matriz_sl2R}
  F=
  \begin{pmatrix}
    \lambda&A&B\\
    A&\mu&C\\
    -B& -C&\tau
  \end{pmatrix},
\end{equation}
in the basis $\{e_k\}_{k=1}^3$ with brackets $[e_1,e_2]=e_3,\, [e_1,e_3]=e_2,\,[e_2,e_3]=-e_1$, is not torsion free.
\item[{\rm (ii)}] The endomorphism $F\colon\mathfrak{so}(3)\to\mathfrak{so}(3)$ defined by the coordinate matrix
\begin{equation}\label{matriz_so3}
  F=
  \begin{pmatrix}
    \lambda&A&B\\
    A&\mu&C\\
    B& C&\tau
  \end{pmatrix},
\end{equation}
in the basis $\{e_k\}_{k=1}^3$ with  brackets $[e_1,e_2]=-e_3,\, [e_1,e_3]=e_2,\,[e_2,e_3]=-e_1$, is not torsion free.
\end{enumerate}
\end{lemma}

\begin{proof}
We prove only the first statement as the second follows an analogous argument. In the first case, by~\eqref{torsion_tensor} we have that the torsion free condition for $F$ is equivalent to
\begin{align*}
0=N_F(e_1,e_2)&= 2(AC-B\mu)e_1+2(AB-\lambda C)e_2-(1+A^2+B^2+C^2-\lambda\mu+\lambda\tau+\mu\tau)e_3,\\[5pt]
0=N_F(e_1,e_3)&=-2(A\tau+BC)e_1-(1-A^2-B^2+C^2+\lambda\mu-\lambda\tau+\mu\tau)e_2-2(AB-\lambda C)e_3,\\[5pt]
0=N_F(e_2,e_3)&=(1-A^2+B^2-C^2+\lambda\mu+\lambda\tau-\mu\tau)e_1+2(A\tau+BC)e_2+2(AC-B\mu)e_3.
\end{align*}

We will arrive to a contradiction assuming that $F$ is torsion free.
From
$$
0=e^3(N_F(e_1,e_2))+e^2(N_F(e_1,e_3))=2(1+C^2+\mu\tau),
$$
we get that $\mu\tau<0$ (in particular $\mu,\tau\not=0$). So, from the equation $0=e^1(N_F(e_1,e_2))=2(AC-B\mu)$ we have $B=\frac{AC}{\mu}$, and substituting in $0=e^1(N_F(e_1,e_3))=-2(A\tau+BC)$ we get $A(C^2+\mu\,\tau)=0$, which implies $A=0$ because $C^2+\mu\,\tau =-1$. As $A=0$, then $B=0$.

Now, from
$$
0=e^2(N_F(e_1,e_3))+e^1(N_F(e_2,e_3))=2(1+\lambda\mu),
$$
we get $\lambda\mu=-1$. In particular $\lambda\not=0$ and we obtain $C=0$ from the equation $e^3(N_F(e_1,e_3))=0$.

Finally, from the equation $e^2(N_F(e_1,e_3))=0$ we get $\tau(\lambda-\mu) = 0$, so $\lambda=\mu$. But this implies $\lambda\mu=\lambda^2> 0$, contradicting that $\lambda\mu=-1$.
\end{proof}

In the following result we consider $\frg=\frg_1\oplus\frg_2$ as any of the $3\oplus3$ decomposable Lie algebras in \eqref{lista-descomp}.  We shall prove that for any closed three-form $\rho$ on $\frg$, if $\lambda(\rho)<0$ then the almost complex structure $J_\rho$ induces an endomorphism on $\frg_1$ which is not torsion free.

\begin{proposition}\label{prop_sl2R}
The $3\oplus3$ decomposable unimodular non-solvable Lie algebras do not admit any complex structure with non-zero closed $(3,0)$-form.
\end{proposition}

\begin{proof}
The $3\oplus3$ decomposable unimodular non-solvable Lie algebras are given in \eqref{lista-descomp}, and they split as $\frg=\frg_1\oplus\frg_2$, where
$\frg_1$ is either $\mathfrak{sl}(2,\R)$ or $\mathfrak{so}(3)$, and $\frg_2$ runs  the following list: $\R^3$, $\frh_3$, $\fre(1,1)$, $\fre(2)$, $\mathfrak{sl}(2,\R)$,
$\mathfrak{so}(3)$. Hence, we distinguish two cases depending on the first summand:

\smallskip

\noindent $\bullet$ The case $\frg=\mathfrak{sl}(2,\R)\oplus\mathfrak{g}_2$.
We show the details of the proof for the Lie
algebra $\mathfrak{g}_2=\R^3$. Firstly, any closed three-form $\rho$ on the Lie algebra $\mathfrak{sl}(2,\R)\oplus\R^3$ is given by
    \begin{align*}
      \rho= \, &\, a_1 e^{123}+a_2 e^{124}+a_3 e^{125}+a_4 e^{126}+a_5 e^{134}+a_6 e^{135}+a_7 e^{136}+a_8 e^{234}+a_9e^{235}\\
      & +a_{10} e^{236}+a_{11} e^{456},
  \end{align*}
where $a_1,\ldots,a_{11}\in\R$. Taking the volume form $\nu=e^{123456}$ we consider the endomorphism $\tilde J_\rho$ defined by \eqref{criterio}.

Suppose that $\tilde \lambda(\rho)<0$, i.e. we get that $J_\rho$ is an almost complex structure on $\frg$. We define the linear endomorphisms $F_\rho=\pi_1\circ J_\rho\circ i_1$ and $\tilde F_\rho=\pi_1\circ \tilde J_\rho\circ i_1$ on $\mathfrak{sl}(2,\R)$ induced by $J_\rho$ and $\tilde J_\rho$, respectively.
Recall that $J_\rho=|\tilde \lambda(\rho)|^{-1/2}\tilde J_\rho$, so $F_\rho$ is also a multiple of $\tilde F_\rho$, i.e.
$$
F_\rho=|\tilde \lambda(\rho)|^{-1/2}\tilde F_\rho.
$$
By a direct calculation one gets that the induced endomorphism
$\tilde F_\rho$ of $\mathfrak{sl}(2,\R)$ is given by~\eqref{matriz_sl2R} with the values $A=B=C=0$ and $\lambda=\mu=\tau=-a_{1}a_{11}$, so $F_\rho\colon \mathfrak{sl}(2,\R) \to \mathfrak{sl}(2,\R)$ is given by the values $\lambda,\mu,\tau,A,B,C$ multiplied by the constant $|\tilde \lambda(\rho)|^{-1/2}$, and therefore $F_\rho$ belongs to the same family of linear endomorhisms defined by~\eqref{matriz_sl2R}. Now, we can apply Lemma~\ref{lema_sl2R} to conclude that $F_\rho$ is not torsion free. This fact, together with Lemma~\ref{lema_Magnin}, implies that $J_\rho$ is not torsion free as well.
In particular, the latter excludes the existence of a complex structure with closed $(3,0)$-form on $\mathfrak{sl}(2,\R)\oplus\R^3$.

For the remaining cases for $\frg_2$, the proofs follow similar arguments but taking into account the corresponding values $\lambda,\mu,\tau$ and $A,B,C$ provided in
Table~\ref{tabla_sl2R} in the Appendix.

\smallskip

\noindent $\bullet$ The case $\frg=\mathfrak{so}(3)\oplus\mathfrak{g}_2$. As in the previous case, we show the details of the proof for the particular Lie
algebra $\mathfrak{g}_2=\mathfrak{h}_3$. A generic closed three-form
$\rho$ on the Lie algebra $\mathfrak{so}(3)\oplus\mathfrak{h}_3$ is given by
  \begin{align*}
  \rho= \, &\, a_1 e^{123} + a_2 e^{124} + a_3 e^{125} + a_4 e^{134} + a_5 e^{135} + a_6 e^{234} + a_7 e^{235} + a_8 (-e^{145} + e^{236}) \\[2pt]
    & +a_9 (e^{136} + e^{245}) + a_{10} (-e^{126} + e^{345}) + a_{11} e^{456},
  \end{align*}
where $a_1,\ldots,a_{11}\in\R$. Consider the endomorphism $\tilde J_\rho$ defined by \eqref{criterio} with respect to the volume form $\nu=e^{123456}$. Then, one gets that the induced endomorphism $\tilde F_\rho=\pi_1\circ \tilde J_\rho\circ i_1\colon \mathfrak{so}(3)\to \mathfrak{so}(3)$ is given by~\eqref{matriz_so3} with the values $A=-2 a_8 a_9$, $B=-2 a_{10} a_8$, $C=2 a_{10} a_9$, and
$$
\lambda=-a_{10}^2 - a_1 a_{11} + a_8^2 - a_9^2,\quad \mu=-a_{10}^2 - a_1 a_{11} - a_8^2 + a_9^2,\quad \tau=a_{10}^2 - a_1 a_{11} - a_8^2 - a_9^2.
$$
Hence, applying Lemma \ref{lema_sl2R} the endomorphism $F_\rho=|\tilde \lambda(\rho)|^{-1/2}\tilde F_\rho$ is not torsion free. Therefore, $N_{J_\rho}$ is not zero and consequently no complex structure with closed $(3,0)$-form exists on the Lie algebra $\mathfrak{so}(3)\oplus\mathfrak{h}_3$.

Similar arguments follow for the remaining cases for $\frg_2$, taking into account the corresponding values $A,B,C,\lambda,\mu,\tau$ provided in
Table~\ref{tabla_su2} in the Appendix.
\end{proof}

The remaining $5\oplus1$ decomposable Lie algebra $A_{5,40}\oplus\R$ requires a more involved argument that relies on the class of $\tilde \lambda(\rho)$ modulo the ideal generated by the polynomials that determine the closedness of the four-form $\tilde J_\rho\rho$.

\begin{proposition}\label{prop_5mas1}
The Lie algebra $A_{5,40}\oplus\R$ does not admit any complex structure with non-zero closed $(3,0)$-form.
\end{proposition}

\begin{proof}
Any generic closed three-form of the Lie algebra $A_{5,40}\oplus\R$ is given by
  \begin{align*}
  \rho= \, & \, a_1 e^{123} + a_2 e^{125} + a_3 e^{126} + a_4 (e^{135}-e^{124}) + a_5 e^{136} + a_6 e^{234} + a_7 (e^{134} + e^{235}) +  a_8 e^{236}\\[2pt]
    & + a_9 (e^{256} - e^{146}) + a_{10}(e^{246} + e^{356})  + a_{11} e^{456},
  \end{align*}
where $a_1,\ldots,a_{11}\in\R$. Taking the volume form $\nu=e^{123456}$, we consider the endomorphism $\tilde J_\rho$ defined by~\eqref{criterio}
and we find that the 4-form $d\big(\tilde J_\rho\rho\big)$ expresses as
    \begin{align*}
    d\big(\tilde J_\rho\rho\big)= \, & \, q_1e^{1234}+q_2e^{1235}+q_3e^{1245}+q_4(e^{1246}-e^{1356})+q_5e^{1256}+q_6e^{1345}+q_7(e^{1346}+e^{2356})\\[2pt]
    & + q_8e^{2345}+q_9e^{2346},
  \end{align*}
where $q_j:=q_j(a_1,\,\ldots,\,a_{11})$ denote the polynomials in the variables $a_1,\,\ldots,\,a_{11}$ given by
  \begin{align*}
    q_1& = -2(a_{10}a_2a_6+a_{10}a_4a_7-2a_4a_6a_9+2a_7^2a_9),\\[3pt]
    q_2& = -2(2a_{10}a_4^2+2a_{10}a_2a_7+a_2a_6a_9+a_4a_7a_9),\\[3pt]
    q_3& = -4a_{11}(a_4^2+a_2a_7),\\[3pt]
    q_4& = -2(a_{10}^2a_2+a_{11}a_4a_5+2a_{11}a_3a_7-a_{11}a_2a_8+a_{10}a_4a_9+2a_7a_9^2),\\[3pt]
    q_5& = -6(a_{11}a_3a_4-a_{11}a_2a_5-a_{10}a_2a_9+a_4a_9^2),\\[3pt]
    q_6& = 2a_{11}(a_2a_6+a_4a_7),\\[3pt]
    q_7& = -2(2a_{10}^2a_4-a_{11}a_3a_6+a_{11}a_5a_7-2a_{11}a_4a_8+a_{10}a_7a_9-a_6a_9^2),\\[3pt]
    q_8& = -4a_{11}(a_4a_6-a_7^2),\\[3pt]
    q_9& = -6(a_{11}a_5a_6+a_{10}^2a_7-a_{11}a_7a_8+a_{10}a_6a_9).
  \end{align*}

The expression of $\tilde \lambda(\rho)$ is a polynomial of degree 4 belonging to the polynomial ring $\R[a_1,\ldots,a_{11}]$ in the
variables $a_1,\ldots,a_{11}$. Let $\mathcal{I}=\langle q_1,\ldots,q_9\rangle$ be the ideal generated by the polynomials $q_1,\ldots,q_9$ above. We need to compute\footnote{
We used the computer software Singular, available at http:/\!/www.singular.uni-kl.de.} the class of $\tilde \lambda(\rho)$ modulo the ideal $\mathcal{I}$.  Concretely,
one gets
$$
\tilde \lambda(\rho)=
(a_1a_{11})^2-a_{10}q_2+a_8q_3-a_7q_4+\frac23a_6q_5+a_5q_6+\frac13a_2q_9.
$$

Now, the cancellation of $d\big(\tilde J_\rho\rho\big)$ clearly requires the cancellation of $q_1,\ldots,q_9$, but the latter implies that
$\tilde \lambda(\rho)=(a_1a_{11})^2\geq0$. Therefore, no complex structure with closed $(3,0)$-form exists on $A_{5,40}\oplus\R$.
\end{proof}

\subsection{The indecomposable case}\label{Sec_indecomposable}

In this section we study the existence of complex structures on the
indecomposable unimodular non-solvable Lie algebras $ L_{6,1}$, $L_{6,2}$ and $L_{6,4}$.

\begin{proposition}\label{resto_Lie_algebras}
The Lie algebras $L_{6,1}$, $L_{6,2}$ and $L_{6,4}$ do not admit complex structures with non-zero closed $(3,0)$-form.
\end{proposition}

\begin{proof}
The proofs for the Lie algebras $L_{6,1}$ and $L_{6,4}$ follow a similar argument to the proof of Proposition~\ref{prop_5mas1} for $A_{5,40}\oplus\R$, but
taking into account the corresponding values charted in Table~\ref{tabla_L61-L64}
in the Appendix for the forms $\rho$ and $d(\tilde J_\rho\rho)$, the polynomials $q_1,\ldots,q_9$ and $\tilde \lambda(\rho)$.

The proof of the statement for $L_{6,2}$ is less straightforward than the former cases, so we will give the details for this algebra. Any closed three-form $\rho$ on $L_{6,2}$ is expressed as
  \begin{align*}
  \rho= \, & \, a_1 e^{123} + a_2 e^{124} + a_3 e^{135} + a_4 (e^{234}-e^{125}) + a_5(e^{235}-e^{134})  + a_6 (e^{236}-e^{145})   \\[2pt]
    & +a_7 (e^{245}-2e^{126}) + a_8 (e^{146}+e^{256}) + a_9 (2e^{136} + e^{345}) + a_{10}(e^{346}-e^{156}) + a_{11} e^{456},
  \end{align*}
where $a_1,\ldots,a_{11}\in\R$.
A direct calculation shows that the 4-form $d\big(\tilde J_\rho\rho\big)$ is given by
  \begin{align*}
    d\big(\tilde J_\rho\rho\big) = \, & \,  q_1e^{1234}+q_2e^{1235}+q_3e^{1245}+q_4e^{1246}+q_5(e^{1256}-e^{2346})+q_6e^{1345} \\[2pt]
    & + q_7(e^{1346}-e^{2356})+q_8e^{1356}+q_9e^{2345},
  \end{align*}
  where $q_j:=q_j(a_1,\,\ldots,\,a_{11})$ denote the polynomials in the variables $a_1,\,\ldots,\,a_{11}$ given by
  \begin{align*}
    q_1 = \, & \, 2(2a_{10}a_4^2+2a_{10}a_2a_5+a_4a_6^2-2a_1a_{10}a_7-3a_5a_6a_7-2a_3a_7^2+a_2a_3a_8-a_4a_5a_8+a_1a_6a_8\\[0pt]
    &\ \ \,  -a_2a_6a_9-2a_4a_7a_9),\\[3pt]
    q_2 =  &   -2(a_{10}a_2a_3-a_{10}a_4a_5+a_1a_{10}a_6+a_5a_6^2+a_3a_6a_7+2a_3a_4a_8+2a_5^2a_8+3a_4a_6a_9\\[0pt]
    &\quad \quad  -2a_5a_7a_9+2a_1a_8a_9-2a_2a_9^2),\\[3pt]
    q_3 = \, & \,  2(2a_{11}a_4^2+2a_{11}a_2a_5+a_{10}a_2a_6-2a1a_{11}a_7+2a_{10}a_4a_7-2a_6^2a_7-3a_4a_6a_8+2a_5a_7a_8\\[0pt]
    &\ \ \,  -a_1a_8^2-8a_7^2a_9+4a_2a_8a_9),\\[3pt]
    q_4 = \, & \,  6(a_{11}a_2a_6+2a_{11}a_4a_7-2a_{10}a_7^2+a_{10}a_2a_8+a_6a_7a_8+a_4a_8^2),\\[3pt]
    q_5 = \, & \,  2(a_2(a_{10}^2-2a_{11}a_9)-a_{11}a_4a_6+4a_{11}a_5a_7-3a_{10}a_6a_7-a_{10}a_4a_8+a_6^2a_8+2a_5a_8^2-2a_7a_8a_9),\\[3pt]
    q_6 =  &   -2(a_1a_{10}^2-2a_{11}a_3a_4-2a_{11}a_5^2+3a_{10}a_5a_6+4a_{10}a_3a_7-a_3a_6a_8-2a_1a_{11}a_9+2a_{10}a_4a_9\\[0pt]
    &\quad \quad  -2a_6^2a_9+2a_5a_8a_9-8a_7a_9^2),\\[3pt]
    q_7 =  &   -2(2a_4(a_{10}^2-2a_{11}a_9)-a_{11}a_5a_6+a_{10}a_6^2+2a_{11}a_3a_7-a_{10}a_5a_8+a_3a_8^2-2a_{10}a_7a_9+3a_6a_8a_9),\\[3pt]
    q_8 =  &   -6(a_5(a_{10}^2-2a_{11}a_9)+a_{11}a_3a_6+a_{10}a_3a_8-a_{10}a_6a_9-2a_8a_9^2),\\[3pt]
    q_9 = \, & \,  2(a_{11}a_2a_3-a_{11}a_4a_5+a_1a_{11}a_6+2a_{10}a_4a_6+a_6^3-3a_{10}a_5a_7+a_1a_{10}a_8+2a_5a_6a_8\\[0pt]
    &\ \ \,  +a_3a_7a_8-a_{10}a_2a_9+4a_6a_7a_9+3a_4a_8a_9).
  \end{align*}

Moreover, for $\tilde \lambda(\rho)$ one gets
\begin{align*}
\tilde \lambda(\rho) = \, & \, (2a_2a_3a_6-2a_4a_5a_6+a_1a_6^2+4a_3a_4a_7+4a_5^2a_7-4a_4^2a_9-4a_2a_5a_9+4a_1a_7a_9+a_1^2a_{11})a_{11}\\[2pt]
&+\frac{3}{2}a_{10}q_1-2a_9q_3+\frac{5}{6}a_3q_4+a_5q_5-a_7q_6+\frac{1}{2}a_4q_7+\frac{2}{3}a_2q_8+\frac{a_6}{2}q_9.
\end{align*}

Next we distinguish two cases depending on the cancellation of the term $a_{10}^2-2a_{11}a_9$ appearing in the polynomials $q_5$, $q_7$ and $q_8$.
In the subsequent analysis
we assume that $a_{11}\neq0$, otherwise $d\big(\tilde J_\rho\rho\big)=0$ would imply that $\tilde \lambda(\rho)=0$.

\smallskip

\noindent {\rm (i)}  If $a_{10}^2-2a_{11}a_9\neq0$, then from the cancellation of $q_5,\,q_7$ and $q_8$ one gets
\begin{align*}
  a_2&=\frac{-a_{11}a_4a_6+4a_{11}a_5a_7-3a_{10}a_6a_7-a_{10}a_4a_8+a_6^2a_8+2a_5a_8^2-2a_7a_8a_9}{2a_{11}a_9-a_{10}^2},\\[5pt]
  a_4&=\frac{-a_{11}a_5a_6+a_{10}a_6^2+2a_{11}a_3a_7-a_{10}a_5a_8+a_3a_8^2-2a_{10}a_7a_9+3a_6a_8a_9}{2(2a_{11}a_9-a_{10}^2)},\\[5pt]
  a_5&=\frac{a_{11}a_3a_6+a_{10}a_3a_8-a_{10}a_6a_9-2a_8a_9^2}{2a_{11}a_9-a_{10}^2}.
\end{align*}
After substituting $a_2,\,a_4$ and $a_5$ in $q_4$ we obtain
\begin{equation*}
q_4=  \frac{3 (a_{10}^3+a_{11}^2 a_3-3 a_{10} a_{11} a_9) (a_{11} a_6^2-2 a_{10}^2 a_7+2 a_{10} a_6 a_8+4 a_{11} a_7 a_9+2 a_8^2 a_9)^2}{(a_{10}^2-2a_{11}a_9)^3}.
\end{equation*}
Now, in order to have $q_4=0$ we consider the following two cases:
\begin{itemize}
  \item $a_3 = \frac{3 a_{10} a_{11} a_9 -a_{10}^3}{a_{11}^2}$, but in this case we have
  \begin{equation*}
    \tilde\lambda(\rho)= \frac{(a_1 a_{11}^2+a_{11} a_6^2-2 a_{10}^2 a_7+2 a_{10} a_6 a_8+4 a_{11} a_7 a_9+2 a_8^2 a_9)^2}{a_{11}^2}\geq0;
  \end{equation*}
  \item $a_{11} a_6^2-2 a_{10}^2 a_7+2 a_{10} a_6 a_8+4 a_{11} a_7 a_9+2 a_8^2 a_9=0$, but this case again yields to a non-negative
  $\tilde\lambda(\rho)=(a_1a_{11})^2$.
\end{itemize}
Therefore, the condition $a_{10}^2-2a_{11}a_9\neq0$ always implies $\tilde\lambda(\rho)\geq0$.

\smallskip

\noindent {\rm (ii)}  Let $a_{10}^2-2a_{11}a_9=0$, so $a_9=\frac{a_{10}^2}{2a_{11}}$. In this case $q_8$ factorizes as
$q_8=\frac{3 (a_{10}^3-2 a_{11}^2 a_3) (a_{11} a_6+a_{10} a_8)}{a_{11}^2}$.
Hence, $q_8=0$ if and only if one of the following two cases holds:
\begin{itemize}
  \item $a_3=\frac{a_{10}^3}{2a_{11}^2}$, which implies that the polynomial $q_7$ factorizes as
  \begin{equation*}
    q_7=\frac{(a_{11} a_6+a_{10} a_8) (2 a_{11}^2 a_5-2 a_{10} a_{11} a_6-a_{10}^2 a_8)}{a_{11}^2}.
  \end{equation*}
  Now,
     if $a_{11} a_6+a_{10} a_8=0$ then we get that $q_6=\frac{(2 a_{11}^2 a_5+a_{10}^2 a_8)^2}{a_{11}^3}$ and
$\tilde\lambda(\rho)=\frac{a_1^2 a_{11}^6+(2 a_{11}^2 a_5+a_{10}^2 a_8)^2 (2 a_{11} a_7+a_8^2)}{a_{11}^4}$.
    Hence, $q_6=0$ implies $\tilde\lambda(\rho)=(a_1a_{11})^2$.

    On the other hand, when $2 a_{11}^2 a_5-2 a_{10} a_{11} a_6-a_{10}^2 a_8=0$ we arrive again at a non-negative $\tilde\lambda(\rho)=
    \frac{(a_1 a_{11}^3+(a_{11} a_6+a_{10} a_8)^2)^2}{a_{11}^4}$.
  \item  $a_6=-\frac{a_{10}a_8}{a_{11}}$, which implies that the polynomial $q_7$ factorizes as $q_7=\frac{(a_{10}^3-2 a_{11}^2 a_3) (2 a_{11} a_7+a_8^2)}{a_{11}^2}$. We can suppose that $a_{10}^3-2 a_{11}^2 a_3\not=0$ (otherwise we lie in the previous case). Then, $q_7=0$ if and only if $2 a_{11} a_7+a_8^2=0$, but the latter yields $\tilde\lambda(\rho)=(a_1a_{11})^2$.
\end{itemize}
Therefore, the condition $a_{10}^2-2a_{11}a_9=0$ also implies $\tilde\lambda(\rho)\geq0$, so the proof of the proposition is complete.
\end{proof}

\begin{remark}\label{L63}
More examples of complex structures with closed $(3,0)$-form can be found in the class of six-dimensional non-solvable Lie algebras if we allow
the Lie algebra to be non-unimodular. Indeed, consider the Lie algebra
$L_{6,3}$ (following the notation in \cite[Table 2]{F-Schu2}) with structure equations
$$
de^1=e^{23},\ de^2=2e^{12},\ de^3=-2e^{13},\ de^4=e^{14}+e^{25}+e^{46},\ de^5=-e^{15}+e^{34}+e^{56},\ de^6=0.
$$
Then, the endomorphism $J$ given by
\begin{equation*}
\begin{array}{lll}
  Je^1=\frac32e^2+\frac16e^3,& Je^2=-\frac13e^1-\frac19e^6,& Je^3=-3e^1+e^6,\\[1em]
  Je^4=\frac13e^5,& Je^5=-3e^4, & Je^6=\frac92e^2-\frac12e^3,
  \end{array}
\end{equation*}
defines an almost complex structure on $L_{6,3}$ such that the nonzero $(3,0)$-form $\Psi=(e^1-iJe^1)\wedge(e^2-iJe^2)\wedge(e^4-iJe^4)$ is closed.
\end{remark}

\begin{remark}\label{para-complex}
Our results have also applications to para-complex and closed $SL(3,\C)$ structures on unimodular non-solvable Lie groups.

First, we recall that similarly to the almost complex case, if a stable 3-form $\rho$ satisfies $\lambda(\rho)>0$ then the endomorphism $J_\rho$ defines a \emph{para-complex} structure on $\frg$, i.e. $J_\rho^2={\rm Id}$
and the eigenspaces for the eigenvalues $\pm 1$ are three-dimensional (see~\cite{Hitchin} for more details).
The form $\rho + {\rm e}\,J_\rho\rho$, where ${\rm e}^2=1$, is a $(3,0)$-form with respect to $J_\rho$.
If this $(3,0)$-form is closed, then the corresponding torsion tensor $N_{J_\rho}$ vanishes identically and the para-complex structure is integrable.
It is clear that any decomposable Lie algebra $\frg=\frg_1\oplus\frg_2$, with $\dim \frg_1=\dim \frg_2=3$, has para-complex structures with closed $(3,0)$-form.
The Lie algebras $A_{5,40}\oplus\R$, $L_{6,1}$, $L_{6,2}$, $L_{6,3}$ and $L_{6,4}$ also admit this type of structures.
In fact, from the proofs of the Propositions~\ref{prop_5mas1} and~\ref{resto_Lie_algebras}, and from Table~\ref{tabla_L61-L64} in the Appendix, it is enough to take $a_1=a_{11}=1$ and the other coefficients $a_j$ equal to zero.
Note that for $L_{6,3}$ the closed 3-form $\rho=e^{123}+e^{456}$ defines a para-complex structure for which $J_\rho\rho$ is also closed (see Remark~\ref{L63} for the structure equations of $L_{6,3}$).

In relation to $\text{SL}(3,\C)$ structures, we recall that an oriented six-dimensional differentiable manifold $M$ admits an $\text{SL}(3,\C)$-structure if its frame
bundle can be reduced to $\text{SL}(3,\C)$. Alternatively, such a structure is defined  by a stable
three-form $\rho\in\Omega^3(M)$ inducing an almost complex structure $J_\rho$, that is, $\lambda(\rho)<0$.
The $\text{SL}(3,\C)$-structure is called \emph{closed} if $d\rho=0$, and in this case
the four-form $d(J_\rho\rho)$ has bidegree $(2,2)$ with respect to $J_\rho$.
Notice that, being $\rho$ closed, the integrability of $J_\rho$ is equivalent to $d(J_\rho\rho)=0$.
For nilpotent Lie algebras, Fino and Salvatore classify in~\cite{FS} the closed $\text{SL}(3,\C)$-structures for which $d(J_\rho\rho)$ is a non-zero (semi-)positive $(2, 2)$-form.
The results in this section could be of interest in relation to the study of such structures in the non-solvable setting.
\end{remark}

\subsection{Classification results}\label{resultados-classif}

The results obtained in Sections~\ref{Sec_decomposable} and~\ref{Sec_indecomposable} are summed up in the following theorem.

\begin{theorem}\label{clasif-non-solvable}
Let $\frg$ be an unimodular non-solvable Lie algebra of dimension 6. Then $\frg$ admits a complex structure with a non-zero closed $(3,0)$-form if and only if it is isomorphic to $\mathfrak{so}(3,1)$.
\end{theorem}

For completeness, in the following corollaries we collect the result obtained in Theorem~\ref{clasif-non-solvable} together with other known results.
Firstly, from Salamon's classification in the nilpotent case \cite{S} and the classification obtained in \cite[Thm. 2.8]{FOU} for solvable Lie algebras, we get

\begin{corollary}\label{classification}
Let $\frg$ be an unimodular Lie algebra of dimension 6. Then, $\frg$ admits a complex structure with a non-zero closed $(3,0)$-form if and only if it is isomorphic to one in the following list:
$$
\begin{array}{rl}
\frh_{1} \!\!&\!\!= (0^6),\\[2pt]
\frh_{2} \!\!&\!\!= (0^4,12,34),\\[2pt]
\frh_{3} \!\!&\!\!= (0^5,12+34),\\[2pt]
\frh_{4} \!\!&\!\!= (0^4,12,14+23),\\[2pt]
\frh_{5} \!\!&\!\!= (0^4,13+42,14+23),\\[2pt]
\frh_{6} \!\!&\!\!= (0^4,12,13),\\[2pt]
\frh_{7} \!\!&\!\!= (0^3,12,13,23),\\[2pt]
\frh_{8} \!\!&\!\!= (0^5,12),\\[2pt]
\frh_{9} \!\!&\!\!= (0^4,12,14+25),\\[2pt]
\frh_{10} \!\!&\!\!= (0^3,12,13,14),\\[2pt]
\frh_{11} \!\!&\!\!= (0^3,12,13,14+23),\\[2pt]
\frh_{12} \!\!&\!\!= (0^3,12,13,24),\\[2pt]
\frh_{13} \!\!&\!\!= (0^3,12,13+14,24),\\[2pt]
\frh_{14} \!\!&\!\!= (0^3,12,14,13+42),
\end{array}
\quad
\begin{array}{rl}
\frh_{15} \!\!&\!\!= (0^3,12,13+42,14+23),\\[2pt]
\frh_{16} \!\!&\!\!= (0^3,12,14,24),\\[2pt]
\frh^-_{19} \!\!&\!\!= (0^3,12,23,14-35),\\[2pt]
\frh^+_{26} \!\!&\!\!= (0^2,12,13,23,14+25),\\[2pt]
\frg_1 \!\!&\!\!= (15,-25,-35,45,0,0) ,\\[2pt]
\frg_2^{\alpha} \!\!&\!\!= (\alpha\!\cdot\!\!15\!+\!25,-\!15\!+\!\alpha\!\cdot\!\!25,-\alpha\!\cdot\!\!35\!+\!45,-\!35\!-\!\alpha\!\cdot\!\!45,0,0),
\\[2pt]
\frg_3 \!\!&\!\!= (0,-13,12,0,-46,-45) ,\\[2pt]
\frg_4 \!\!&\!\!= (23,-36,26,-56,46,0) ,\\[2pt]
\frg_5 \!\!&\!\!= (24+35,26,36,-46,-56,0) ,\\[2pt]
\frg_6 \!\!&\!\!= (24+35,-36,26,-56,46,0) ,\\[2pt]
\frg_7 \!\!&\!\!= (24+35,46,56,-26,-36,0) ,\\[2pt]
\frg_8 \!\!&\!\!= (16-25,15+26,-36+45,-35-46,0,0) ,\\[2pt]
\frg_9 \!\!&\!\!= (45,15+36,14-26+56,-56,46,0) ,\\[2pt]
\mathfrak{so}(3,\!1) \!\!&\!\!= (23\!-\!56,-13\!+\!46,12\!-\!45,26\!-\!35,-16\!+\!34,15\!-\!24).
\end{array}
$$
\end{corollary}

Here $\alpha\geq 0$ is any non-negative real number. The Lie algebras in the list are pairwise non-isomorphic. The Lie algebras $\frh_k$ are nilpotent, $\frg_l$ are solvable, and $\mathfrak{so}(3,\!1)$ is the only semi-simple Lie algebra in the classification.
On the other hand, the decomposable Lie algebras are the following:
$\frh_1$ ($1\oplus \cdots \oplus 1$), $\frh_2, \frg_3$ ($3\oplus 3$), $\frh_3, \frh_6, \frh_9, \frh_{16},\frg_1,\frg_2^{\alpha}$ ($1\oplus 5$), and $\frh_8$
($1\oplus 1\oplus 1\oplus 3$).

\medskip

Let $X$ be a complex manifold with $\dim_{\C}X=n$. A Hermitian metric $F$ on $X$ is said to be \emph{balanced} if $dF^{n-1}=0$. Important
aspects of these metrics were first investigated by Michelsohn~\cite{Michelsohn}. The balanced Hermitian geometry plays a central role in heterotic string theory, as the next sections show.

We recall that any left-invariant Hermitian metric $F$ on a complex Lie group is balanced, so the Lie algebra $\mathfrak{so}(3,\!1)$ admits balanced Hermitian structures. The $6$-dimensional nilpotent, resp. unimodular solvable with closed $(3,0)$-form, Lie algebras admitting balanced Hermitian structures are classified in \cite[Thm. 26]{Ug}, resp. in \cite[Thm. 4.5]{FOU}. As a second consequence of the results obtained in the previous sections we have

\begin{corollary}\label{balanced-classification}
Let $\frg$ be an unimodular Lie algebra of dimension 6. Then, $\frg$ admits a complex structure with a non-zero closed $(3,0)$-form having balanced metrics if and only if it is isomorphic to $\frh_{1}$, $\frh_{2}$, $\frh_{3}$, $\frh_{4}$, $\frh_{5}$, $\frh_{6}$, $\frh^-_{19}$, $\frg_1$, $\frg_2^{\alpha} (\alpha\geq 0)$, $\frg_3$, $\frg_5$, $\frg_7$, $\frg_8$, or $\mathfrak{so}(3,\!1)$.
\end{corollary}

Concerning the existence of lattices (of maximal rank), the connected and simply connected nilpotent Lie group $H_k$ corresponding to the Lie algebra $\frh_k$ admits a lattice by the well-known Malcev theorem.
Moreover, by \cite[Prop. 2.10]{FOU}, the connected and simply connected Lie group $G_l$ corresponding to the solvable Lie algebra $\frg_l$ admits a lattice for any $l\not=2$, whereas for $l=2$
there exists a countable number of distinct $\alpha$'s, including $\alpha=0$, for
which the Lie group $G_2^\alpha$ corresponding to $\frg_2^{\alpha}$ admits a lattice.
For the case of $\mathfrak{so}(3,\!1)$, it is well-known that a lattice exists.
Hence, we get the following result for compact homogeneous spaces in six dimensions:

\begin{corollary}\label{quotient-balanced}
Let $M=\Gamma\backslash G$ be a six-dimensional compact manifold defined as the quotient of a simply connected Lie group $G$ by a lattice $\Gamma$.
Suppose that $M$ possesses an invariant complex structure $J$ with non-zero closed $(3,0)$-form  admitting a balanced metric $F$.
Then, the Lie algebra $\frg$ of $G$
is isomorphic to $\frh_{1},\ldots,\frh_{6}$, $\frh^-_{19}$, $\frg_1$, $\frg_2^{\alpha}$ for some $\alpha\geq 0$, $\frg_3$, $\frg_5$, $\frg_7$, $\frg_8$, or $\mathfrak{so}(3,\!1)$.
\end{corollary}

We recall that such homogeneous spaces are Kähler only when $\frg\cong \frh_{1}$
or $\frg_2^0$ (see Proposition~\ref{solv-g2-0-balanced-real-basis} below for more details). We also recall that the existence of a balanced metric on $(M,J)$ implies the existence of an invariant one (by symmetrization).

\section{The Hull-Strominger system on compact balanced homogeneous spaces}\label{HS-homogeneous}

\noindent
The heterotic superstring
background with non-zero torsion was investigated, independently, by A. Strominger~\cite{Str} and C. Hull~\cite{Hull}, giving rise to a complicated system of partial differential equations.
Their approach allowed to extend the initial proposal for a superstring compactification given in \cite{CHSW} to a complex (non necessarily K\"ahler) setting with trivial canonical bundle.
In dimension six, the system requires the geometric inner space $X$ to be a compact
complex
conformally balanced manifold with holomorphically trivial canonical bundle, which is equipped with an instanton
compatible with the Green-Schwarz anomaly cancellation condition:
\begin{equation}\label{anomaly-canc}
dT=2\pi^2 \alpha' \big(p_1(\nabla)-p_1(A)\big).
\end{equation}
This equality, also known as the Bianchi identity, is an equation of 4-forms, where
$T$ is the torsion of the (Strominger-)Bismut connection of the conformally balanced metric, $\alpha'$ is a real constant (called the slope parameter), and
$p_1(\nabla)$, resp. $p_1(A)$, denotes the 4-form representing
the first Pontrjagin
class of some metric connection $\nabla$, resp. of the instanton $A$.

More concretely, let $(M,J,g,\Psi)$ be a compact manifold of real dimension 6 endowed with a Hermitian structure $(J,g)$ and a non-vanishing holomorphic $(3,0)$-form $\Psi$. Denoting by $F$ the fundamental form, the torsion 3-form is $T=JdF$.
The Hull-Strominger system is given by
\begin{equation}\label{HS}
\begin{array}{ll}
& d(\parallel\!\!\Psi\!\!\parallel_F  \!F^2)= 0,\\[7pt]
& \Omega^{A}\wedge F^2=0,\quad (\Omega^{A})^{0,2}=(\Omega^{A})^{2,0}=0,\\[6pt]
& dT-\frac{\alpha'}{4}\left({\rm tr}\, \Omega\wedge\Omega -{\rm tr}\, \Omega^A\wedge \Omega^A
\right)=0.
\end{array}
\end{equation}
The first equation (the conformally balanced condition) is a reformulation, due to Li and Yau \cite{Li-Yau}, of
the so-called dilatino and gravitino equations of the system. It implies the existence of a balanced Hermitian metric $\hat{F}$ on $(M,J)$ simply
by modifying $F$ conformally as $\hat{F}=\parallel\!\!\Psi\!\!\parallel_F^{1/2}  \!F$.
The second equation, also known as the gaugino equation, is the Hermitian-Yang-Mills equation for the connection $A$, where $\Omega^A$ denotes its curvature. 
The third equation is the anomaly cancellation equation \eqref{anomaly-canc} after taking
${1\over 8\pi^2} {\rm tr}\, \Omega\wedge\Omega$ as the 4-form representing the class $p_1$, where $\Omega$ is the curvature of the connection.

Li and Yau \cite{Li-Yau} found the first non-K\"ahler solutions to the Hull-Strominger system on a K\"ahler Calabi-Yau manifold. Based on a construction by Goldstein and Prokushkin in
\cite{GP}, Fu and Yau first proved the existence of solutions on non-K\"ahler Calabi-Yau inner spaces given as a
$\mathbb{T}^2$-bundle over a K3 surface~\cite{Fu-Yau}.
Recently, the Fu-Yau solution 
is generalized to torus bundles over K3 orbifolds in \cite{FGV}.

An important result by Ivanov \cite{Iv} (see also \cite{FIUV09}) states that a solution of the Hull-Strominger system satisfies in addition the heterotic equations
of motion if and only if the connection $\nabla$ in the
tangent bundle is Hermitian-Yang-Mills, i.e.
\begin{equation}\label{HYM}
\Omega\wedge F^2=0,\quad \Omega^{0,2}=\Omega^{2,0}=0.
\end{equation}
Homogeneous solutions to the system \eqref{HS}-\eqref{HYM} were first found in \cite{FIUV09} on a nilmanifold, and more recently on a solvmanifold and on the quotient of SL(2,$\C$) in \cite{OUV17} (see also \cite{FeiYau}).
On the other hand, the Li-Yau solutions to the Hull-Strominger system given in \cite{Li-Yau} were further extended in~\cite{A-Gar} by a perturbative method for certain K\"ahler Calabi-Yau threefolds with stable holomorphic vector bundle
to prove existence of solutions to the heterotic equations of motion.
New examples of solutions of the system \eqref{HS}-\eqref{HYM} on non-Kähler torus bundles over K3 surfaces 
are constructed by Garcia-Fernandez in \cite{MGarcia-Crelle}, including applications to the construction of T-dual solutions.
More recently, with the ansatz that the metric connection $\nabla$ in the
tangent bundle is Hermitian-Yang-Mills,
a family of Futaki invariants obstructing the existence of solutions of the Hull-Strominger system in a given balanced class $\frb$ is found in \cite{GFGM}.

The results in \cite{Li-Yau,Fu-Yau,FGV} mentioned above require $\nabla$ to be taken as the Chern connection $\nabla^c$. In addition to $\nabla^c$, other metric connections
proposed for the anomaly cancellation equation 
are the Strominger-Bismut connection $\nabla^+$,
the Levi-Civita connection $\nabla^{LC}$ or the Hull connection $\nabla^-$.
The physical and geometrical meaning of different choices for $\nabla$ is discussed by De la Ossa and Svanes in \cite{OS}.
In \cite{FeiYau} the canonical 1-parameter family
of Hermitian connections $\nabla^{\tau}$
found by Gauduchon \cite{Gau} is considered.
The family $\nabla^{\tau}$ contains the Chern connection ($\tau=1$)
and the Bismut connection ($\tau=-1$).
Furthermore, all the previous connections can be gathered in a plane of \emph{metric} linear connections $\nre$ which were introduced and studied in \cite{OUV17},
so that $\nabla^{LC}=\nabla^{0,0}$, $\nabla^{\pm}=\nabla^{\pm\frac12,0}$,
$\nabla^c=\nabla^{0,\frac12}$, and where the Gauduchon connections $\nabla^{\tau}$ correspond to the line $\rho=\frac12-\varepsilon$
(where $\tau=1-4\varepsilon$).
The covariant derivative of this two-parameter family of connections is also considered in \cite{COMS} to study the geometry of a fibration ${\mathbb X}$ over the moduli space of heterotic structures ${\mathcal M}$, where the fibres are 3-folds $X$ together with their metrics and complex structures.

\smallskip

We recall the definition of the connections $\nre$.
Let $(M, J, g)$ be a Hermitian manifold.
A linear connection $\nabla$ defined on the tangent bundle $TM$ is called \textit{Hermitian} if $\nabla g=0$ and $\nabla J=0$, i.e. both the metric and the complex structure are parallel.
For any $(\varepsilon, \rho) \in \mathbb{R}^2$, the connection $\nre$ is defined as
\begin{equation}\label{family_2par}
g(\nabla^{\varepsilon,\rho}_XY, Z) = g(\nabla^{LC}_X Y, Z) + \varepsilon\,T(X,Y,Z) + \rho\,C(X,Y,Z), \quad\ X,Y,Z\in \mathfrak{X}(M)
\end{equation}
where $C(\cdot, \cdot, \cdot)=dF(J\cdot, \cdot, \cdot)$ denotes the torsion of the Chern connection $\nabla^c$, and
$T(\cdot, \cdot, \cdot)=JdF(\cdot, \cdot, \cdot)$
stands for
the torsion $3$-form of the Bismut connection $\nabla^+$. Here $F(\cdot, \cdot) = g(\cdot, J\cdot)$ is the fundamental $2$-form.

\begin{proposition}\cite[Prop. 2.1]{OUV17}
Let $(M,J,g)$ be a Hermitian manifold. For every $(\varepsilon, \rho) \in \mathbb{R}^2$, the connection $\nre$
satisfies the following properties:
$$
\nre g=0,\quad\qquad \nre J=(1-2\varepsilon-2\rho)\,\nabla^{LC}\!J.
$$
Therefore, if $(M,J,g)$ is not K\"ahler then, $\nre$ is Hermitian if and only if $\rho=\frac12-\varepsilon$.
\end{proposition}

The 1-parameter family $\nabla^{\varepsilon,\frac12-\varepsilon}$
is precisely the family of canonical Hermitian connections $\nabla^{\tau}$ found by Gauduchon in \cite{Gau}, where the parameters are related by $\varepsilon=\frac{1-\tau}{4}$ and $\rho=\frac{1+\tau}{4}$.

\smallskip

The invariant Hermitian geometry of compact balanced non-Kähler homogeneous spaces allows to construct explicit solutions of the heterotic equations of motion (and more generally of the Hull-Strominger system) when $\nabla=\nre$ is taken in the anomaly cancellation equation. Note that in the invariant setting, the function $\parallel\!\!\Psi\!\!\parallel_F$ is constant and
the first equation in \eqref{HS} is equivalent to the closedness of $F^{2}$,
i.e. $F$ is balanced.
In fact, in \cite{FIUV09,OUV17} such solutions were found, even with positive slope parameter $\alpha'$, on a nilmanifold with underlying Lie algebra $\frh_3$, a solvmanifold with underlying algebra $\frg_7$, and on the quotient of the semisimple group SL(2,$\C$). Of course, the solutions required to find connections in the $(\re)$-plane satisfying the instanton equation \eqref{HYM}.

In the following three propositions, we recall the results obtained in \cite{OUV17} concerning the Hermitian-Yang-Mills condition for the metric $(\re)$-connections.
In what follows, we will use the notation 
$\nreJF$
to emphasize that we are considering any balanced Hermitian structure $(J,F)$ on the given homogeneous space and any connection in the corresponding $(\re)$-plane.

\smallskip

\begin{proposition}\label{h3}\cite[Prop. 3.2]{OUV17}
Let $(J,F)$ be any invariant balanced Hermitian structure on a $6$-dimensional
nilmanifold~$M$ with underlying nilpotent Lie algebra isomorphic to $\frh_{3}$. Then, the connection $\nreJF$ is an instanton if and only if $(\re)=(\frac12,0)$, i.e. it is the Bismut connection.
\end{proposition}

\begin{proposition}\label{g7}\cite[Prop. 5.2]{OUV17}
Let $J$ be any invariant complex structure with non-zero closed $(3,0)$-form on a $6$-dimensional solvmanifold~$M$ with underlying solvable Lie algebra isomorphic to $\frg_{7}$. Let $F$ be any invariant balanced metric on $(M,J)$, and let $\nreJF$ be a connection in the $(\re)$-plane.
Then, we have:
\begin{enumerate}
\item[{\rm (i)}] If $(\re)\not=(\frac12,0)$, then $\nreJF$ is not an instanton.
\item[{\rm (ii)}] There exist balanced metrics $F$ for which the Bismut connection $\nabla^{\frac12,0}_{\!(\!J,F)}$ is an instanton.
\end{enumerate}
\end{proposition}

\begin{proposition}\label{sl2C}\cite[Prop. 4.1]{OUV17}
Let $J$ be the complex parallelizable structure on a compact quotient $M$ with underlying Lie algebra $\mathfrak{so}(3,\!1)$ (that is, $M$ is the quotient of the semisimple group $SL(2, \C)$ by a lattice of maximal rank). Then, there is a one-parametric family of invariant balanced Hermitian metrics $F_t$ for which the connection $\nabla_{\!\!(\!J,F_t\!)}^{\varepsilon,\rho}$ is an instanton only for $(\re)\in\{(\frac12,0),(0,\frac12)\}$, i.e. for the Chern connection (which is flat) and the Bismut connection.
\end{proposition}

In all the cases, the curvature of the Bismut connection has non-vanishing trace, in particular it is not flat.
Moreover, it was conjectured that these are the only spaces admitting such solutions, more concretely (see \cite[Section 7]{OUV17}), if a compact non-K\"ahler homogeneous space $M=\Gamma\backslash G$
admits an invariant solution of the heterotic equations of motion with $\alpha'>0$ and
with respect to some non-flat connection $\nabla$ in the ansatz $\nre$,
then
$\nabla$ is the Bismut connection and $M$ is one of the spaces above.

\smallskip

Our goal in the second part of the paper is to prove the following result related to this conjecture, which is valid independently of the sign of the slope parameter $\alpha'$.

\begin{theorem}\label{main-instanton}
Let $M=\Gamma\backslash G$ be a six-dimensional compact manifold defined as the quotient of a simply connected Lie group $G$ by a lattice $\Gamma$ of maximal rank.
Suppose that $M$ possesses an invariant balanced Hermitian structure $(J,F)$ with invariant non-zero closed $(3,0)$-form.  
Let $\nreJF$ be a metric connections in the $(\re)$-plane.
If $\nreJF$ is a non-flat instanton, then the Lie algebra of $G$ is isomorphic to $\frh_3$, $\frg_7$, or $\mathfrak{so}(3,\!1)$. 
\end{theorem}

\vskip.1cm

By the classification results obtained in the previous section, it suffices to focus on homogeneous spaces based on the Lie algebras $\frg\not\cong \frh_3$, $\frg_7$, $\mathfrak{so}(3,\!1)$. Thus, we will study the balanced Hermitian geometry on the Lie algebras $\frh_{2}$, $\frh_{4}$, $\frh_5$, $\frh_{6}$, $\frh_{19}^-$, $\frg_1$, $\frg_2^0$, $\frg_2^{\alpha}$ $(\alpha>0)$, $\frg_3$, $\frg_5$, and $\frg_8$. Section~\ref{instantones-nil} is devoted to the nilpotent case, whereas in  Section~\ref{instantones-solv} we study the class of solvmanifolds.
Section~\ref{nuevas-soluciones} explores the role played by flat instantons (Chern connection) in the construction of solutions on the Nakamura manifold with given balanced class.

From now on, the manifold $M=\Gamma\backslash G$ will be a compact quotient of a Lie group $G$ by a lattice $\Gamma$,
endowed with an invariant Hermitian structure $(J,F)$,
that is, $(J,F)$ can be defined at the level of the Lie algebra $\frg$ of $G$.
We will say that a basis $\{e^k\}_{k=1}^6$ for $\frg^*$ is {\it adapted} to the  Hermitian structure if both the complex structure $J$
and the 2-form $F$ express in the canonical way
\begin{equation}\label{adapted-basis}
Je^1=-e^2,\ Je^3=-e^4,\ Je^5=-e^6,\quad\quad F=e^{12}+e^{34}+e^{56}.
\end{equation}
Hence, the metric $g$ is given by $g=e^1\otimes e^1 + \cdots +e^6\otimes e^6$.

Given any linear connection $\nabla$, the connection 1-forms
$(\sigma^{\nabla})^i_j$ with respect to an adapted basis are
$$
(\sigma^{\nabla})^i_j(e_k) = g(\nabla_{e_k}e_j,e_i),
$$
and the curvature 2-forms $(\Omega^{\nabla})^i_j$ are given by
$$
(\Omega^{\nabla})^i_j = d (\sigma^{\nabla})^i_j + \sum_{1\leq k \leq 6}
(\sigma^{\nabla})^i_k\wedge (\sigma^{\nabla})^k_j.
$$

Let $c_{ij}^k$ be the structure constants of the Lie algebra $\frg$ with respect to an adapted basis $\{e^k\}_{k=1}^6$, that is,
$$
d e^k = \sum_{1\leq i<j \leq 6} c_{ij}^k \, e^{i j},\quad\quad
1\leq k\leq 6.
$$
Since $d e^k(e_i,e_j)= -e^k([e_i,e_j])$ and the basis $\{e^k\}_{k=1}^6$ is orthonormal, the Levi-Civita connection
1-forms $(\sigma^{LC})^i_j$ of the metric $g$ express as
$
(\sigma^{LC})^i_j(e_k) = \frac12(c^i_{jk}-c^k_{ij}+c^j_{ki})
$.

\smallskip

Now, let $\nabla=\nreJF$ be any metric connection in the $(\re)$-plane.
Using \eqref{family_2par}, its connection 1-forms are given by
\begin{eqnarray}\label{connection-1-forms}
\nonumber
\big(\sigma^{\nreJF}\big)^i_j (e_k)&=&(\sigma^{LC})^i_j(e_k) - \varepsilon\, T(e_i,e_j,e_k) -\rho\, C(e_k, e_i, e_j) \\
&=&
\nonumber \frac12(c^i_{jk}-c^k_{ij}+c^j_{ki}) -\varepsilon\,JdF(e_i, e_j, e_k) - \rho\, dF(J e_k,e_i,e_j).
\end{eqnarray}

For any $1\leq i,j,k,l\leq6$, we define
$$
\begin{array}{rl}
& \treJF(i,j):=\big(\Omega^{\nreJF}\big)^i_j(e_1,e_2)+\big(\Omega^{\nreJF}\big)^i_j(e_3,e_4)+\big(\Omega^{\nreJF}\big)^i_j(e_5,e_6),  \\[8pt]
& \ureJF(i,j,k,l):=\big(\Omega^{\nreJF}\big)^i_j(Je_k,Je_l)-\big(\Omega^{\nreJF}\big)^i_j(e_k,e_l).
\end{array}
$$

Then, one can express the Hermitian-Yang-Mills condition \eqref{HYM} as follows:

\begin{lemma}\label{condiciones-de-instanton}
The connection $\nreJF$ is an instanton
if and only if for an adapted basis $\{e^k\}_{k=1}^6$ the
following conditions
\begin{equation}\label{Theta-Upsilon}
\treJF(i,j)=0,\quad\quad
\ureJF(i,j,k,l)=0,
\end{equation}
are satisfied  for every $1\leq i,j,k,l\leq6$.
\end{lemma}

\begin{proof}
It follows directly from the conditions $\Omega^{\nreJF}\wedge F^2=0$ and $\big(\Omega^{\nreJF}\big)^{0,2}=\big(\Omega^{\nreJF}\big)^{2,0}=0$, taking into account \eqref{adapted-basis}.
\end{proof}

\section{Instantons on balanced nilmanifolds}\label{instantones-nil}

\noindent In this section we prove Theorem~\ref{main-instanton} for nilmanifolds. We will use the adapted frames obtained in \cite{UV15} for any invariant balanced Hermitian structure.

Recall that by our discussion in Section~\ref{HS-homogeneous} we need to study the balanced Hermitian geometry on the nilmanifolds with underlying Lie algebra isomorphic to any of the nilpotent Lie algebras $\frh_{2}$, $\frh_{4}$, $\frh_5$, $\frh_{6}$, or $\frh_{19}^-$.

In the following result we study the Hermitian geometry on the nilpotent Lie algebra $\frh_5$ endowed with its complex parallelizable structure.

\begin{proposition}\label{nilv-familia-2-14}
Let $(M,J)$ be the Iwasawa manifold and let $F$ be any invariant Hermitian metric on $(M,J)$.
Then, the connection $\nreJF$ is an instanton if and only if
$(\varepsilon,\rho)=(0,\frac12)$, i.e. it is the Chern connection (which is flat).
\end{proposition}

\begin{proof}
Since $J$ is complex parallelizable, it is well-known that any invariant Hermitian metric $F$ on the Iwasawa manifold $(M,J)$ is balanced and the Chern connection is flat, so it is an instanton.

Let  us consider any other connection $\nreJF$.
By \cite[Thm. 2.11]{UV15}, given any such $(J,F)$, there is a basis $\{e^k\}_{k=1}^6$ of 1-forms
satisfying \eqref{adapted-basis} and the following equations
\begin{equation}\label{h5real}
de^1 = de^2=de^3=de^4=0,\quad
de^5 = t\,e^{13} - t\,e^{24} ,\quad
de^6 = t\,e^{14} + t\,e^{23},
\end{equation}
where $t\in\mathbb{R}^*$. By Lemma~\ref{condiciones-de-instanton}, the connection is an instanton if and only if the conditions \eqref{Theta-Upsilon} hold. A direct calculation gives the following particular equations:
\begin{equation*}
\begin{array}{lcl}
\treJF(5,6)=0 & \Longleftrightarrow & 1+2\varepsilon-2 \rho=0,\\[6pt]
\ureJF(1,3,1,3)=0 & \Longleftrightarrow & 1-2 \varepsilon-2\rho=0,
\end{array}
\end{equation*}
which are never satisfied if $(\varepsilon,\rho)\not=(0,\frac12)$. In conclusion, only the Chern connection is an instanton.
\end{proof}

In the following result we exclude the case $\frh_3$ because it is studied in \cite{OUV17} (see Proposition~\ref{h3} above).
This is the reason for taking $(\varrho,b)\not=(0,0)$ in the following result.

\begin{proposition}\label{nilv-familia-2-15}
Let $(J,F)$ be an invariant balanced Hermitian structure on a
nilmanifold~$M$ with underlying Lie algebra isomorphic to $\frh_{2}$, $\frh_{4}$, $\frh_5$ or $\frh_{6}$. Suppose $J$ is not complex parallelizable. Then, the connection $\nreJF$ is never an instanton.
\end{proposition}

\begin{proof}
We use the description of the balanced geometry on the Lie algebras $\frh_{2}$, $\frh_{4}$, $\frh_5$ and $\frh_{6}$ provided in \cite[Thm. 2.11]{UV15}.
For any invariant balanced Hermitian structure $(J,F)$, there is a basis $\{e^k\}_{k=1}^6$ of 1-forms
satisfying \eqref{adapted-basis} and one of the two following sets of equations:
\begin{equation}\label{2-stepreal}
\begin{cases}
\begin{array}{lcl}
de^1 \zzz & = &\zzz de^2=de^3=de^4=0,\\[4pt]
de^5 \zzz & = &\zzz \frac{t}{s}(\varrho+b^2)e^{13} -\frac{t}{s}(\varrho-b^2) e^{24} ,\\[5pt]
  de^6 \zzz & = &\zzz -2\,t\,(e^{12}-e^{34}) + \frac{t}{s}(\varrho-b^2)e^{14}
  + \frac{t}{s}(\varrho+b^2)e^{23};
\end{array}
\end{cases}
\end{equation}
\begin{equation}\label{realchangeAbelianh5}
\begin{cases}
\begin{array}{lcl}
de^1\zzz & = &\zzz de^2=de^3=de^4=0,\\[6pt]
de^5\zzz & = &\zzz s Y \left[ 2 b^2 u_1 |u|\,(e^{12}-e^{34}) - b^2 t
u_1
|u| Y\,(e^{13} + e^{24}) + 2 \varrho s u_1\,(e^{13} - e^{24}) \right.\\[4pt]
&& \quad \left. + 2 s u_2\left((\varrho-b^2)e^{14}+(\varrho+b^2)e^{23}\right) \right],\\[6pt]
de^6 \zzz & = &\zzz s Y \left[ 2 (2s^2- b^2 u_2)
|u|\,(e^{12}-e^{34})+ b^2 t u_2 |u| Y\,(e^{13}+e^{24})
- 2 \varrho s u_2\,(e^{13}-e^{24})\right.\\[4pt]
&& \quad \left. + 2 s
u_1\left((\varrho-b^2)e^{14}+(\varrho+b^2)e^{23}\right) \right].
\end{array}
\end{cases}
\end{equation}
Here $\varrho\in\{0,1\}$, $b\in \mathbb{R}$ with $(\varrho,b)\not=(0,0)$, and $s,t\in\mathbb{R}^*$. In the equations \eqref{realchangeAbelianh5}, $u=u_1+i\,u_2\in \mathbb{C}^*$ with $s^2>|u|^2$, and we are denoting $Y\!:=\frac{2 \sqrt{s^2-|u|^2}}{t\,|u|}$.
We recall that all the complex structures~$J$ are nilpotent, and the
abelian ones correspond to taking $\rho=0$ in~\eqref{2-stepreal} or~\eqref{realchangeAbelianh5}. Since we are considering $(\varrho,b)\not=(0,0)$,  when $J$ is abelian the Lie algebra is $\frh_5$ and we can normalize $b$ so that  $b=1$.

Next, we study the instanton condition for any connection $\nreJF$.
First, we consider the equations~\eqref{2-stepreal}.
The instanton conditions \eqref{Theta-Upsilon} give in particular the following equations:
\begin{equation*}
\begin{array}{lcl}
\treJF(1,2)=0 & \Longleftrightarrow & b^4 (1 + 4(\varepsilon - \rho)^2 ) + 4s^2\left( \rho^2 + (\varepsilon-\frac12)^2 \right) + 4\,\varepsilon\, \varrho^2 (2 \rho -1)=0,\\[6pt]
\ureJF(1,2,1,4)=0 & \Longleftrightarrow & \varrho(\varepsilon - \rho - \frac32)(\varepsilon - \rho +\frac12)=0,\\[6pt]
\ureJF(3,6,1,6)=0 & \Longleftrightarrow & (1+2\varepsilon-2\rho)\left[\varrho(1+2\varepsilon-2\rho)-b^2(1-2\varepsilon-2\rho)\right]=0,\\[6pt]
\ureJF(4,5,1,6)=0 & \Longleftrightarrow & (1+2\varepsilon-2\rho)\left[\varrho(1+2\varepsilon-2\rho)+b^2(1-2\varepsilon-2\rho)\right]=0.
\end{array}
\end{equation*}
Recall that $\varrho\in\{0,1\}$.
If $\varrho=1$, then the equation $\ureJF(1,2,1,4)=0$ implies $\varepsilon-\rho\in \{-\frac12,\frac32\}$. In the first case we have that $\Theta^{\varepsilon,\varepsilon+\frac12}_{\!(\!J,F)}(1,2)\not=0$ since $b^4 + s^2 + 4 \varepsilon^2 (1 + s^2)>0$. On the other hand, if $\varepsilon-\rho=\frac32$ then at least one of the conditions $\Upsilon^{\varepsilon,\varepsilon-\frac32}_{\!(\!J,F)}(3,6,1,6)=0=\Upsilon^{\varepsilon,\varepsilon-\frac32}_{\!(\!J,F)}(4,5,1,6)$ fail.

Let us suppose now that $\varrho=0$, that is to say, the complex structure $J$ is abelian, so we can take $b=1$. But in this case $\treJF(1,2)\not=0$ clearly. Hence, no connection can be an instanton.

\medskip

From now on, we consider the equations \eqref{realchangeAbelianh5} and divide our study into three cases depending on the values of the pair $(\varrho,b)$:

\smallskip

\noindent $\bullet$ If $(\varrho,b)=(0,1)$, then one can check that the equation $\treJF(5,6)=0$ is satisfied if and only if $(1-2 \varepsilon) \rho=0$, so $\varepsilon=\frac12$ or $\rho=0$. A direct calculation of the term $\ureJF(1,3,1,3)$ allows to show that in both cases the equation $\ureJF(1,3,1,3)=0$ holds if and only if
$(\varepsilon,\rho)=(\frac12,0)$. But for the latter connection one gets $\Theta^{\frac12,0}_{\!(\!J,F)}(1,2)\not=0$. Thus, no connection $\nreJF$ is an instanton when $(\varrho,b)=(0,1)$.

\smallskip

\noindent $\bullet$  In the case $(\varrho,b)=(1,0)$ we get the following conditions:
\begin{equation*}
\begin{array}{lcl}
\ureJF(1,5,3,5)=0 & \Longleftrightarrow & (1+2 \varepsilon-2\rho)^2 u_1 =0,\\[6pt]
\ureJF(1,5,3,6)=0 & \Longleftrightarrow & (1+2 \varepsilon-2\rho)^2 u_2 =0.
\end{array}
\end{equation*}
Using that $u\not=0$, one has $\rho=\varepsilon+\frac12$. Now, a direct calculation shows that $\Upsilon^{\varepsilon,\varepsilon+\frac12}_{\!(\!J,F)}(1,3,1,3)=0$ implies $\varepsilon=0$,
which in turn gives $\treJF(5,6)\not=0$. So, there are not instantons for~$(\varrho,b)=(1,0)$.

\smallskip

\noindent $\bullet$ Finally, let us study the case $\varrho=1$ and $b\not=0$. One has:
\begin{equation*}
\begin{array}{lcl}
\treJF(1,4)=0 & \Longleftrightarrow &
\left[(1+2 \varepsilon-2\rho)^2+4\rho(1-2\varepsilon)\right] u_1=0,\\[6pt]
\ureJF(1,2,1,4)=0 & \Longleftrightarrow & (1+2 \varepsilon-2\rho) (3-2 \varepsilon+2\rho)\, u_1 =0,\\[6pt]
\ureJF(1,5,1,6)=0 & \Longleftrightarrow & (1+2 \varepsilon-2\rho) (1-2 \varepsilon+2\rho)\, u_1 =0.
\end{array}
\end{equation*}

\noindent {\bf -}  If $u_1\not=0$, then necessarily $\rho=\varepsilon+\frac12$, and the first equation reduces to $(\varepsilon+\frac12)(\frac12-\varepsilon)=0$, so giving the two possible values $(\varepsilon,\rho)=(\frac12,1)$ or $(\varepsilon,\rho)=(-\frac12,0)$.
One can check that in both cases the condition $\treJF(1,2)=0$ is satisfied if and only if $(b^2-u_2)^2 + 2 s^2 - u_2^2 +1=0$. However, the latter is not possible since $s^2 > |u|^2$ implies $2s^2-u_2^2>0$. 

\smallskip

\noindent {\bf -}  Let us suppose that $u_1=0$ (thus, $u_2\not=0$). The difference $\ureJF(1,5,3,6)-\ureJF(2,5,3,5)$ is a nonzero multiple of $(1+2 \varepsilon-2\rho)^2 (2s^2-b^2u_2)$, so in what follows we will distinguish two cases depending on the vanishing of $2s^2-b^2u_2$:
\begin{enumerate}
\item[{\bf --}] If $2s^2\not=b^2 u_2$, then $\rho=\varepsilon+\frac12$. Moreover, we get
\begin{equation*}
\begin{array}{lcl}
\Theta^{\varepsilon,\varepsilon+\frac12}_{\!(\!J,F)}(1,3)=0 & \Longleftrightarrow &
 (1+2 \varepsilon) (1-2 \varepsilon) (2 s^2-b^2 u_2)=0,\\[6pt]
\Upsilon^{\varepsilon,\varepsilon+\frac12}_{\!(\!J,F)}(1,3,1,3)=0 & \Longleftrightarrow & \varepsilon \left( (b^2-u_2)^2 + 2 s^2 - u_2^2 +1 \right)=0.
\end{array}
\end{equation*}
Arguing as above, there are not solutions for this system.
\item[{\bf --}] Finally, we consider the case $2s^2=b^2 u_2$, so we can take  $u_2=2s^2/b^2$. Note that the condition $s^2>|u|^2$ translates into the inequality $b^4-4s^2>0$. A direct calculation gives the following equations:
\begin{equation*}
\begin{array}{lrl}
\phantom{espacio} \treJF(1,2)\!-\treJF(3,4)=0 \!\!&\!\! \Longleftrightarrow \!\!&\!\!
(1 + 2 \varepsilon)^2\!-\!16\varepsilon\rho+4\rho^2=0,\\[6pt]
\phantom{espacio} \treJF(5,6)=0 \!\!&\!\! \Longleftrightarrow \!\!&\!\! (1 + 2 \varepsilon)^2 \!-\!16\varepsilon\rho+4\rho^2 \!-\! 4(1 \!-\! 2 \varepsilon) \rho (1 + b^4 \!-\! 2 s^2)=0.
\end{array}
\end{equation*}
Taking the difference we get $(1 - 2 \varepsilon) \rho (1 + b^4 - 2 s^2)=0$, therefore $\varepsilon=\frac12$ or $\rho=0$ (since $b^4 - 4 s^2>0$).
Now, if $\varepsilon=\frac12$ the condition $\Theta^{\frac12,\rho}_{\!(\!J,F)}(5,6)=0$ reduces to $(1-\rho)^2 (b^4-4 s^2) =0$, which implies $\rho=1$; however, in such case one can check that $\Theta^{\frac12,1}_{\!(\!J,F)}(1,2)\not=0$.
Similarly, if $\rho=0$ then $\Theta^{\varepsilon,0}_{\!(\!J,F)}(5,6)=0$ reduces to
$(1+2 \varepsilon)^2 (b^4-4 s^2) =0$, thus $\varepsilon=-\frac12$, but again $\Theta^{-\frac12,0}_{\!(\!J,F)}(1,2)\not=0$.
\end{enumerate}

In conclusion, no connection $\nreJF$ is an instanton when $\varrho=1$ and $b\not=0$, and the proof of the proposition is complete.
\end{proof}

In the following result we consider the nilmanifolds having $\frh_{19}^-$ as 
underlying Lie algebra.

\begin{proposition}\label{nilv-h19}
Let $(J,F)$ be an invariant balanced Hermitian structure on a
nilmanifold~$M$ with underlying Lie algebra isomorphic to $\frh_{19}^-$.
Then, the connection $\nreJF$ is never an instanton.
\end{proposition}

\begin{proof}
The description of the balanced geometry on the Lie algebra $\frh_{19}^-$ can be found in \cite[Thm. 2.11]{UV15}. For any invariant balanced Hermitian structure $(J,F)$, there is a basis $\{e^k\}_{k=1}^6$ of 1-forms
satisfying \eqref{adapted-basis} and one of the two following sets of equations:
\begin{equation}\label{str-eq-Family-I}
\begin{cases}
\ de^1 = de^2=0,\quad de^3 = \frac{2s}{r}\,e^{15},\quad  de^4 =
\frac{2s}{r}\,e^{25},\quad de^5=0,\quad de^6 =\pm\frac{2}{rs}\,(e^{13}+e^{24});
\end{cases}
\end{equation}
\begin{equation}\label{str-eq-Family-II}
\begin{cases}
\begin{array}{lcl}
de^1 \zzz & = &\zzz de^2=0,\\[4pt] de^3
\zzz & = &\zzz\frac{s}{rtZ}\,
\left[\pm\frac{t^2}{s^2}\,(e^{13}+e^{24})\pm \frac{t^2}{s^2}(st+Z)\,
(e^{25}-e^{16}) + e^{14}+ \frac{1}{st+Z}\,e^{15}\right],
\\[4pt] de^4 \zzz & = &\zzz\frac{s}{rtZ}\,\left[e^{24}+
\frac{1}{st+Z}\,e^{25}\right],
\\[4pt] de^5\zzz & = &\zzz\frac{-s}{rtZ}\,\left[
(st+Z)\,e^{24}+e^{25}\right],
\\[4pt] de^6 \zzz & = &\zzz\frac{s}{rtZ}\,
\left[\pm \frac{t^2}{s^2}
\frac{1}{st+Z}\,(e^{13}+e^{24})\pm\frac{t^2}{s^2}\,(e^{25}-e^{16}) +
(st+Z)\,e^{14}+ e^{15}\right].
\end{array}
\end{cases}
\end{equation}
Here $r,s,t\in\mathbb{R}^*$, and in the equations \eqref{str-eq-Family-II} we have $s^2t^2>1$ and $Z:=\sqrt{s^2t^2-1}$.

We recall that any complex structure $J$ on $\frh_{19}^-$ is non-nilpotent, and there are
only two complex structures $J_0^{\pm}$ up to isomorphism. The $\pm$-sign in the equations above corresponds to $J=J_0^{\pm}$, respectively.

For \eqref{str-eq-Family-I}, a direct calculation shows that
\begin{equation*}
\begin{array}{lcl}
\treJF(1,2)=0 & \Longleftrightarrow & \left((\varepsilon-\frac12)^2+\rho^2\right)(1+s^4)=0,\\[6pt]
\treJF(5,6)=0 & \Longleftrightarrow &
\rho (\varepsilon-\frac12) + \varepsilon (\rho-\frac12) s^4 = 0,
\end{array}
\end{equation*}
whereas for the equations \eqref{str-eq-Family-II} we get
\begin{equation*}
\begin{array}{lcl}
\treJF(1,2)=0 & \Longleftrightarrow & ((\varepsilon-\frac12)^2+\rho^2)(s^4+t^4)=0,\\[6pt]
\treJF(3,5)=0 & \Longleftrightarrow &
\left((\varepsilon + \frac12)^2 + 2\varepsilon + \rho(\rho - 4\varepsilon - 2)\right)s^4-\left((\varepsilon-\frac12)^2+\rho(\rho-4\varepsilon+2)\right)t^4=0.
\end{array}
\end{equation*}
It is clear that in both cases the respective system of equations does not have any solutions in $(\varepsilon,\rho)$, hence \eqref{Theta-Upsilon} in Lemma~\ref{condiciones-de-instanton} is never satisfied and the connection $\nreJF$ is not an instanton.
\end{proof}

\section{Instantons on balanced solvmanifolds}\label{instantones-solv}

\noindent In this section we prove Theorem~\ref{main-instanton} for solvmanifolds.
Recall that, by the discussion in Section~\ref{HS-homogeneous}, we need to study the balanced Hermitian geometry on the solvmanifolds whose underlying algebra is isomorphic to any of the following solvable Lie algebras: $\frg_1$, $\frg_2^0$, $\frg_2^{\alpha >0}$, 
$\frg_3$, $\frg_5$, or $\frg_8$.

A first difference with the nilpotent setting is that there are not adapted frames available in the literature for the solvable Lie algebras, so we will find them for each case.

Along this section, $J$ will always refer to an invariant complex structure with non-zero closed $(3,0)$-form.
Given a (1,0)-basis $\{\omega^k\}_{k=1}^3$ for the complex structure $J$, any invariant Hermitian metric $F$ expresses as
\begin{equation}\label{2form}
2\,F=i\,(r^2\omega^{1\bar1}+s^2\omega^{2\bar2}+t^2\omega^{3\bar3})+u\omega^{1\bar2}-\bar
u\omega^{2\bar1}+v\omega^{2\bar3}-\bar
v\omega^{3\bar2}+z\omega^{1\bar3}-\bar
z\omega^{3\bar1},
\end{equation}
where the coefficients $r^2,\,s^2,\,t^2$ are non-zero real numbers and $u,\, v,\, z\in\C$ satisfy $r^2s^2>|u|^2$, $s^2t^2>|v|^2$, $r^2t^2>|z|^2$ and $r^2s^2t^2 + 2\Real(i\bar u\bar v z)>t^2|u|^2 + r^2|v|^2 + s^2|z|^2$.
The balanced condition for $F$ imposes restrictions on these metric coefficients, which need to be considered in order to find an adapted basis for $(J,F)$.
In what follows, we will use the term ``\emph{diagonal}'' to refer to a balanced metric $F$ with $u=v=z=0$ in its expression \eqref{2form} with respect to a certain (1,0)-basis.

\smallskip

We will start with the classification of complex structures obtained in \cite{FOU} to find adapted frames on each balanced Hermitian solvmanifold, and then we will study the instanton condition for any connection $\nreJF$ in the associated $(\re)$-plane.

\smallskip

The following lemma will be used when studying solvmanifolds with underlying Lie algebra isomorphic to $\frg_1$, $\frg_2^\alpha$ or $\frg_8$.

\begin{lemma}\label{lema-solv}
Let $\frg$ be a $6$-dimensional Lie algebra endowed with a complex structure $J$
defined by the complex equations
\begin{equation}\label{jotas-solv-KL}
d\omega^1=K\,\omega^{13}+ L\,\omega^{1\bar{3}},\quad \
d\omega^2=-K\,\omega^{23}- L\,\omega^{2\bar{3}},\quad \
d\omega^3=0,
\end{equation}
where $K,L \in \C$ with $L\not=0$.
Let $F$ be any Hermitian metric given by \eqref{2form}. Then, we have:

\vskip.2cm

\noindent {\rm (a)}
The metric $F$ is balanced if and only if $v=z=0$.

\vskip.2cm

\noindent {\rm (b)} For any balanced metric $F$,
there is a (1,0)-basis $\{\tau^{k}\}_{k=1}^3$ such that
$F=\frac{i}{2}
(\tau^{1\bar1}+\tau^{2\bar2}+\tau^{3\bar3})$,

and
$$
d\tau^1= \frac{K}{t} \tau^{13} +  \frac{L}{t} \tau^{1\bar{3}},\quad
d\tau^2 = \frac{2uK}{t\,\sqrt{1-|u|^2}} \tau^{13} +  \frac{2uL}{t\,\sqrt{1-|u|^2}} \tau^{1\bar{3}} - \frac{K}{t} \tau^{23} -  \frac{L}{t} \tau^{2\bar{3}},\quad
d\tau^3 = 0.
$$

\vskip.2cm

\noindent {\rm (c)}
The real basis $\{e^1,\ldots,e^6\}$ for $\frg^*$ defined by $e^{2k-1}+i\,e^{2k}:=\tau^k$,
$1\leq k\leq 3$,
satisfies \eqref{adapted-basis}.

\end{lemma}

\begin{proof}
The metric $F$ is balanced if and only if $F^2$ is a closed form, equivalently $\db F\wedge F=0$. Using \eqref{jotas-solv-KL} we get
$$
4\,\db F\wedge F=L\,(u\bar{z}-ir^2\bar{v})\omega^{13\bar{1}\bar{2}\bar{3}} -L\,(is^2\bar{z}+\bar{u}\bar{v})\omega^{23\bar{1}\bar{2}\bar{3}}.
$$
Since $L\not=0$, $F$ is balanced if and only if $u\bar{z}-ir^2\bar{v}=0=is^2\bar{z}+\bar{u}\bar{v}$. Now, $r^2s^2>|u|^2$ implies that the latter conditions are equivalent to $v=z=0$. This proves (a).

Notice that we can normalize the metric coefficients $r$ and $s$, so that any  balanced Hermitian structure $(J,F)$ still has a (1,0)-basis satisfying~\eqref{jotas-solv-KL} and the metric writes as
\begin{equation}\label{2form-simplificada}
2\,F=i\,\omega^{1\bar1}+i\,\omega^{2\bar2}+i\,t^2\omega^{3\bar3}+u\,\omega^{1\bar2}-\bar u\,\omega^{2\bar1}, \quad t\in \R^*, \quad u\in \B=\{ u\in\C\mid |u|<1\}.
\end{equation}
Hence, $F$ can be written as
$$
2\,F=i\,(1-|u|^2)\, \omega^{1\bar1}
+i \left( u\, \omega^{1} +i\,\omega^{2}\right)\!\wedge\!( \bar{u}\, \omega^{\bar 1} -i\,\omega^{\bar 2} ) + i t^2\omega^{3\bar3}
=\
i(\tau^{1\bar1}+\tau^{2\bar2}+\tau^{3\bar3}),
$$
where  $\{\tau^{k}\}_{k=1}^3$ is the (1,0)-basis defined by
\begin{equation}\label{cambio-para-v-z-0}
\tau^{1}=\sqrt{1-|u|^2}\, \omega^{1}, \quad
\tau^{2}=u\, \omega^{1} +i\,\omega^{2}, \quad
\tau^{3}=t\,\omega^{3}.
\end{equation}
A direct calculation shows that the complex structure equations \eqref{jotas-solv-KL} express in this basis as in~(b).

Finally, (c) follows from (b) by considering the real and imaginary parts of $\tau^k$, $1\leq k\leq 3$, i.e.
$e^1+i\,e^2:=\tau^1$, $e^3+i\,e^4:=\tau^2$ and $e^5+i\,e^6:=\tau^3$.
\end{proof}

\smallskip

In the following result we begin with solvmanifolds with $\frg_1$ as underlying Lie algebra.

\begin{proposition}\label{solv-g1-balanced-real-basis}
Let $M$ be a $6$-dimensional
solvmanifold with underlying solvable Lie algebra isomorphic to $\frg_1$.
For every invariant balanced Hermitian structure $(J,F)$, there is a basis $\{e^k\}_{k=1}^6$ of 1-forms on $M$
satisfying \eqref{adapted-basis} and the following equations:
\begin{equation}\label{solv-g1}
\begin{cases}
\begin{array}{lcl}
de^1 \zzz & = &\zzz \frac{2}{t} e^{15},\\[4pt]
de^2 \zzz & = &\zzz \frac{2}{t} e^{25},\\[4pt]
de^3 \zzz & = &\zzz \frac{4 u_1}{t \sqrt{1-|u|^2}} e^{15}-\frac{4 u_2}{t \sqrt{1-|u|^2}} e^{25}-\frac{2}{t} e^{35} ,\\[5pt]
 de^4 \zzz & = &\zzz \frac{4 u_2}{t \sqrt{1-|u|^2}} e^{15}+\frac{4 u_1}{t \sqrt{1-|u|^2}} e^{25}-\frac{2}{t} e^{45} ,\\[5pt]
de^5 \zzz & = &\zzz de^6=0,
\end{array}
\end{cases}
\end{equation}
where $t\in \R^*$ and $u=u_1+i\,u_2 \in \B=\{ u\in\C\mid |u|<1\}$.

\smallskip

Moreover, the connection $\nreJF$ is an instanton if and only if $(\varepsilon,\rho)=(0,\frac12)$ and $u=0$; in other words, the (non-K\"ahler) balanced metric $F$ is ``diagonal'' and $\nreJF$ is precisely the Chern connection (which is flat).
\end{proposition}

\begin{proof}
By \cite[Prop. 3.3]{FOU}, up to isomorphism, there is only one complex structure $J$ with closed $(3,0)$-form on the Lie algebra $\frg_1$. Its complex equations are given by \eqref{jotas-solv-KL} with $K=L=1$.
Then, by Lemma~\ref{lema-solv} we consider the basis of invariant 1-forms $\{e^k\}_{k=1}^6$ on $M$ defined in (c). Now, by a direct calculation from the complex equations given in Lemma~\ref{lema-solv}~(b) for $K=L=1$ we arrive at \eqref{solv-g1}.

\smallskip

Next, we study the instanton condition.
One can check that the equation $\treJF(5,6)=0$ is satisfied if and only if $\varepsilon (1-2 \rho) =0$, so $\varepsilon=0$ or $\rho=\frac12$.
But, if $\varepsilon=0$ then $\Upsilon^{0,\rho}_{\!(\!J,F)}(1,3,1,3)=0$ implies $1-2 \rho =0$, so $\rho=\frac12$.
On the other hand, if $\rho=\frac12$ then $\Upsilon^{\varepsilon,\frac12}_{\!(\!J,F)}(1,3,1,3)=0$ implies
$\varepsilon=0$. So, in any case we are reduced to $(\re)=(0,\frac12)$, i.e. the Chern connection.

Now, let us take $(\re)=(0,\frac12)$. In this case the instanton condition is equivalent to the vanishing of the terms $\Theta^{0,\frac12}_{\!(\!J,F)}(i,j)$, for $1\leq i<j\leq 4$.  Moreover, $\Theta^{0,\frac12}_{\!(\!J,F)}(1,2)=-\Theta^{0,\frac12}_{\!(\!J,F)}(3,4)$, $\Theta^{0,\frac12}_{\!(\!J,F)}(1,3)=\Theta^{0,\frac12}_{\!(\!J,F)}(2,4)$ and $\Theta^{0,\frac12}_{\!(\!J,F)}(1,4)=-\Theta^{0,\frac12}_{\!(\!J,F)}(2,3)$, hence the instanton condition \eqref{Theta-Upsilon}
is equivalent to the vanishing of the following three terms
$$
\Theta^{0,\frac12}_{\!(\!J,F)}(1,2)= \frac{-8 \,|u|^2}{t^2 (1-|u|^2)}, \quad\
\Theta^{0,\frac12}_{\!(\!J,F)}(1,3)= \frac{-8 \,u_2}{t^2 \sqrt{1-|u|^2}}, \quad\
\Theta^{0,\frac12}_{\!(\!J,F)}(1,4)= \frac{8 \,u_1}{t^2 \sqrt{1-|u|^2}},
$$
which in turn is equivalent to $u=0$ (i.e. the balanced metric $F$ is diagonal when written as \eqref{2form-simplificada} in the basis $\{\omega^{k}\}_{k=1}^3$).
Since $dF=-\frac{4}{t} e^{125} +\frac{4}{t} e^{345}$, the metric $F$ is not K\"ahler.

Finally, when $u=0$ all the curvature forms of the Chern connection vanish identically, so the instanton is flat.
\end{proof}

In the following result we consider solvmanifolds with underlying Lie algebra $\frg_2^0$. Recall that these solvmanifolds admit invariant Hermitian metrics which are K\"ahler.

\begin{proposition}\label{solv-g2-0-balanced-real-basis}
Let $M$ be a solvmanifold with underlying Lie algebra isomorphic to $\frg_2^0$.
For every invariant balanced Hermitian structure $(J,F)$, there is a basis $\{e^k\}_{k=1}^6$ of 1-forms on $M$
satisfying \eqref{adapted-basis} and the equations
\begin{equation}\label{solv-g2-0}
\begin{cases}
\begin{array}{lcl}
de^1 \zzz & = &\zzz -\frac{2}{t} e^{25},\\[4pt]
de^2 \zzz & = &\zzz \frac{2}{t} e^{15},\\[4pt]
de^3 \zzz & = &\zzz -\frac{4 u_2}{t \sqrt{1-|u|^2}} e^{15}-\frac{4 u_1}{t \sqrt{1-|u|^2}} e^{25}+\frac{2}{t} e^{45} ,\\[5pt]
 de^4 \zzz & = &\zzz \frac{4 u_1}{t \sqrt{1-|u|^2}} e^{15}-\frac{4 u_2}{t \sqrt{1-|u|^2}} e^{25}-\frac{2}{t} e^{35} ,\\[6pt]
de^5 \zzz & = &\zzz de^6=0,
\end{array}
\end{cases}
\end{equation}
where $t\in \R^*$ and $u=u_1+i\,u_2\in \B=\{ u\in\C\mid |u|<1\}$.

Moreover, if the balanced metric is not K\"ahler then $\nreJF$ is never an instanton.

The K\"ahler metrics correspond to $u=0$ in \eqref{solv-g2-0}; in this case the connection $\nreJF$, which coincides with the Levi-Civita connection, has all the curvature forms identically zero.
\end{proposition}

\begin{proof}
By \cite[Prop. 3.3]{FOU} there is only one complex structure $J$, up to isomorphism, with closed $(3,0)$-form on the Lie algebra $\frg_2^0$, with complex equations given by \eqref{jotas-solv-KL} for $K=L=i$.
Using Lemma~\ref{lema-solv}, we consider the basis of invariant 1-forms $\{e^k\}_{k=1}^6$ on $M$ defined in (c). Now, \eqref{solv-g2-0} follows by a direct calculation from the complex structure equations given in Lemma~\ref{lema-solv}~(b) with $K=L=i$.

\smallskip

Next, we study the instanton condition for any connection $\nreJF$.
Firstly, we note that the equation
$\treJF(1,2)=-4 \frac{((1-2 \varepsilon)^2+4 \rho^2) |u|^2}{t^2 (1-|u|^2)}=0$ implies that $(\varepsilon,\rho)=(\frac12,0)$,
or $u=0$.

In the case $(\varepsilon,\rho)=(\frac12,0)$, the condition
$\Theta^{\frac12,0}_{\!(\!J,F)}(5,6)=\frac{16|u|^2}{t^2 (1-|u|^2)}=0$ implies that $u=0$.
So, there are not instantons when the metric is non-K\"ahler.

Clearly, when $u=0$, since $F$ is K\"ahler, the connections are all the Levi-Civita connection.
In this case one has that all the curvature forms vanish identically.
\end{proof}

In the following result we study solvmanifolds with underlying Lie algebra isomorphic to $\frg_2^\alpha$, for some $\alpha>0$.
Recall that there is a countable number of distinct $\alpha$'s for
which the connected and simply connected solvable Lie group corresponding to $\frg_2^\alpha$ admits a lattice \cite[Prop. 2.10]{FOU}.

\begin{proposition}\label{solv-g2-alfa-balanced-real-basis}
Let $M$ be a solvmanifold with underlying Lie algebra isomorphic to $\frg_2^\alpha$, for some $\alpha>0$.
For every invariant balanced Hermitian structure $(J,F)$ on $M$, there is a basis of 1-forms $\{e^k\}_{k=1}^6$
satisfying \eqref{adapted-basis} and the following equations:
\begin{equation}\label{solv-g2-alfa}
\begin{cases}
\begin{array}{lcl}
de^1 \zzz & = &\zzz \frac{2\delta}{t} \cos \theta\, e^{15} -\frac{2}{t} \sin \theta\, e^{25},\\[4pt]
de^2 \zzz & = &\zzz  \frac{2}{t} \sin \theta\, e^{15} +\frac{2\delta}{t} \cos \theta\, e^{25},\\[4pt]
de^3 \zzz & = &\zzz \frac{4 \delta u_1 \cos \theta -4 u_2 \sin \theta}{t \sqrt{1-|u|^2}} e^{15} - \frac{4 \delta u_2 \cos \theta +4 u_1 \sin \theta}{t \sqrt{1-|u|^2}} e^{25} - \frac{2\delta}{t} \cos \theta\, e^{35} +\frac{2}{t} \sin \theta\, e^{45},\\[5pt]
 de^4 \zzz & = &\zzz  \frac{4 \delta u_2 \cos \theta +4 u_1 \sin \theta}{t \sqrt{1-|u|^2}}  e^{15} + \frac{4 \delta u_1 \cos \theta -4 u_2 \sin \theta}{t \sqrt{1-|u|^2}}  e^{25} - \frac{2}{t} \sin \theta\, e^{35} -\frac{2\delta}{t} \cos \theta\, e^{45},\\[5pt]
de^5 \zzz & = &\zzz de^6=0,
\end{array}
\end{cases}
\end{equation}
where $\delta=\pm 1$ and $\theta\in \left(0,\pi/2\right)$, and where
$t\in \R^*$ and $u=u_1+i\,u_2\in \B=\{ u\in\C\mid |u|<1\}$.

\smallskip

Moreover, the connection $\nreJF$ is an instanton if and only if $(\varepsilon,\rho)=(0,\frac12)$ and $u=0$; that is, the (non-K\"ahler) balanced metric $F$ is ``diagonal'' and $\nreJF$ is the Chern connection (which is flat).
\end{proposition}

\begin{proof}
Any complex structure $J$ with closed $(3,0)$-form on the Lie algebra
$\frg_2^\alpha$ is given, up to isomorphism, by the equations \eqref{jotas-solv-KL} for $K=L=\delta\cos \theta+i\sin \theta$,
where $\delta=\pm 1$ and $\theta\in \left(0,\pi/2\right)$.
Note that $\alpha=\cos \theta/\sin \theta$.
Then, by Lemma~\ref{lema-solv} we consider the basis of invariant 1-forms $\{e^k\}_{k=1}^6$ on $M$ defined in (c). Now, by a direct calculation from the complex equations given in Lemma~\ref{lema-solv}~(b) for $K=L=\delta\cos \theta+i\sin \theta$ we arrive at \eqref{solv-g2-alfa}.

\smallskip

Now, we will find the connections $\nreJF$ that satisfy the instanton condition  \eqref{Theta-Upsilon}.
Consider in particular the following conditions:
\begin{equation*}
\begin{array}{lcl}
\ureJF(1,3,1,3)=0 & \Longleftrightarrow & \frac{ (1+2\varepsilon-2 \rho) (1-2\varepsilon-2\rho) ( \cos^2\!\theta+ |u|^2 \sin^2\!\theta)}
{t^2 (1-|u|^2)}=0,\\[6pt]
\treJF(5,6)=0 & \Longleftrightarrow &
\frac{ \varepsilon (1-2 \rho) ( \cos^2\!\theta+ |u|^2 \sin^2\!\theta)}
{t^2 (1-|u|^2)} =0.
\end{array}
\end{equation*}
Using that $\theta\in \left(0,\pi/2\right)$, we get $(1-2 \rho)^2-4\varepsilon^2=0=\varepsilon(1-2 \rho)$.
Hence, $(\re)=(0,\frac12)$, that is, we are reduced to study the Chern connection.
But, now we have
$\Theta^{0,\frac12}_{\!(\!J,F)}(1,2)=\frac{-8 |u|^2}{t^2 (1-|u|^2)} =0$,
which implies that $u=0$, i.e. the metric is diagonal (when written in \eqref{2form-simplificada} with respect to the basis $\{\omega^{k}\}_{k=1}^3$).
Note that $dF=-\frac{4\delta\,\cos \theta}{t} (e^{12} - e^{34})\wedge e^5\not=0$, so the diagonal metric is not K\"ahler.

Finally, when $u=0$ all the curvature forms of the Chern connection vanish identically, so it is a flat instanton.
\end{proof}

\begin{remark}\label{comentario-Pujia}
Note that $\frg_1$ and $\frg_2^\alpha$ are precisely the solvable Lie algebras in Corollary~\ref{classification} which have a codimension-one abelian ideal. A Lie group $G$ whose Lie algebra has this property is called \emph{almost-abelian}.
Pujia studies in \cite{Pujia} the instanton condition for the Gauduchon connections of  invariant balanced structures on $6$-dimensional almost-abelian Lie groups,
so the Propositions~\ref{solv-g1-balanced-real-basis}, \ref{solv-g2-0-balanced-real-basis} and~\ref{solv-g2-alfa-balanced-real-basis} extend such study to the whole $(\re)$-plane of metric connections $\nre$.
\end{remark}

In the following two propositions we prove that the solvmanifolds with underlying Lie algebra isomorphic to $\frg_3$ or $\frg_5$ do not provide any instanton.

\begin{proposition}\label{solv-g3-balanced-real-basis}
Let $M$ be a solvmanifold with underlying Lie algebra isomorphic to $\frg_3$.
For every invariant balanced Hermitian structure $(J,F)$, there is
a basis $\{e^k\}_{k=1}^6$ of 1-forms on $M$
satisfying \eqref{adapted-basis} and the following equations:
\begin{equation}\label{solv-g3}
\begin{cases}
\begin{array}{lcl}
de^1 \zzz & = &\zzz d e^2=0,\\[4pt]
de^3 \zzz & = &\zzz  \frac{v_2}{r t^2}\, e^{13} + \frac{v_1}{r t^2}\, e^{14}
-\frac{\sqrt{s^2 t^2-|v|^2}}{r t^2}\, e^{16},\\[4pt]
de^4 \zzz & = &\zzz  \frac{2 x v_2}{r t^2}\, e^{13} + \frac{2 x v_1}{r t^2}\, e^{14}
-\frac{2 x \sqrt{s^2 t^2-|v|^2}}{r t^2}\, e^{16}
+  \frac{v_2-2 x v_1}{r t^2}\, e^{23}
 + \frac{v_1+2 x v_2}{r t^2}\, e^{24} \\[4pt]
\zzz &  &\zzz +  \frac{2 x \sqrt{s^2 t^2-|v|^2}}{r t^2}\, e^{25}
- \frac{\sqrt{s^2 t^2-|v|^2}}{r t^2}\, e^{26},\\[5pt]
de^5 \zzz & = &\zzz  \frac{t^4 + 2 x v_2(v_1-2 x v_2)}{2 x r t^2 \sqrt{s^2 t^2-|v|^2}}\, e^{13} + \frac{v_1(v_1-2 x v_2)}{r t^2 \sqrt{s^2 t^2-|v|^2}}\, e^{14}
-\frac{v_1-2 x v_2}{r t^2}\, e^{16}\\[5pt]
\zzz &  &\zzz
-  \frac{t^4+v_2^2-2 x v_1 v_2}{r t^2 \sqrt{s^2 t^2-|v|^2}}\, e^{23}
 + \frac{t^4 - 2 x v_2(v_1+2 x v_2)}{2 x r t^2 \sqrt{s^2 t^2-|v|^2}}\, e^{24}  -  \frac{2 x v_2}{r t^2}\, e^{25}
+ \frac{v_2}{r t^2}\, e^{26},\\[6pt]
de^6 \zzz & = &\zzz \frac{t^4 + v_2(v_2+2 x v_1)}{r t^2 \sqrt{s^2 t^2-|v|^2}}\, e^{13} + \frac{v_1(v_2+2 x v_1)}{r t^2 \sqrt{s^2 t^2-|v|^2}}\, e^{14}
-\frac{v_2+2 x v_1}{r t^2}\, e^{16}\\[5pt]
\zzz &  &\zzz
+  \frac{v_1(v_2-2 x v_1)}{r t^2 \sqrt{s^2 t^2-|v|^2}}\, e^{23}
 + \frac{v_1(v_1+2 x v_2)}{r t^2 \sqrt{s^2 t^2-|v|^2}}\, e^{24}  +  \frac{2 x v_1}{r t^2}\, e^{25}
- \frac{v_1}{r t^2}\, e^{26},
\end{array}
\end{cases}
\end{equation}
where $x\in \mathbb{R}^+$, and where $r,s,t\in\mathbb{R}^*$ and $v=v_1+i\,v_2\in \mathbb{C}$ satisfy $s^2 t^2>|v|^2$.

\smallskip

Moreover, $\nreJF$ is never an instanton.
\end{proposition}

\begin{proof}
As proved in \cite[Prop. 3.4]{FOU}, any complex structure $J$ with closed $(3,0)$-form on the Lie algebra
$\frg_3$ is given, up to isomorphism, by the equations
\begin{equation}\label{jotas-solv-g3}
d\omega^1\!=0,\ \ \
d\omega^2\!=-\frac12 \omega^{13} - \frac{1+2xi}{2} \omega^{1\bar{3}}+xi\,\omega^{3\bar{1}},\ \ \
d\omega^3\!=\frac12 \omega^{12} +\frac{2x-i}{4x}\omega^{1\bar{2}}
+\frac{i}{4x}\,\omega^{2\bar{1}},
\end{equation}
where $x\in \mathbb{R}^+$.

An invariant Hermitian metric $F$ is balanced if and only if it is given by \eqref{2form} with $u=z=0$ (see the proof of \cite[Thm. 4.5]{FOU} for details).
Consider the (1,0)-basis $\{\tau^{k}\}_{k=1}^3$ defined by
\begin{equation}\label{cambio-para-u-z-0}
\tau^{1}= r \,\omega^{1}, \quad
\tau^{2}=\frac{\Delta}{t}\, \omega^{2}, \quad
\tau^{3}=\frac{v}{t}\, \omega^{2} +i t\,\omega^{3},
\end{equation}
where $\Delta=\sqrt{s^2t^2-|v|^2}$. Hence,
the metric can be written as $F=\frac{i}{2}(\tau^{1\bar1}+\tau^{2\bar2}+\tau^{3\bar3})$, and
the complex structure equations \eqref{jotas-solv-g3} express in this basis as
\begin{equation}\label{jotas-solv-g3-taus}
\begin{cases}
\begin{array}{lcl}
d\tau^1 \zzz & = &\zzz 0,\\[5pt]
d\tau^2 \zzz & = &\zzz -\frac{v\,i}{2rt^2} \tau^{12} + \frac{\Delta\,i}{2rt^2} \tau^{13} - \frac{(2x-i)\bar{v}}{2rt^2} \tau^{1\bar{2}} + \frac{(2x-i)\Delta}{2rt^2} \tau^{1\bar{3}} -\frac{xv}{rt^2} \tau^{2\bar{1}}
+\frac{x\Delta}{rt^2} \tau^{3\bar{1}},\\[6pt]
d\tau^3 \zzz & = &\zzz \frac{(t^4-v^2)\,i}{2rt^2\Delta} \tau^{12} + \frac{v\,i}{2rt^2} \tau^{13}  +\frac{t^4- 4x^2|v|^2 + 2x (|v|^2+ t^4)\,i}{4xrt^2\Delta}  \tau^{1\bar{2}} + \frac{(2x-i)v}{2rt^2} \tau^{1\bar{3}} \\[6pt]
&&  -  \frac{4x^2v^2+t^4}{4xrt^2\Delta}  \tau^{2\bar{1}}
+\frac{xv}{rt^2} \tau^{3\bar{1}}.
\end{array}
\end{cases}
\end{equation}
Taking the real basis $\{e^k\}_{k=1}^6$ as in Lemma~\ref{lema-solv}~(c)
and using \eqref{jotas-solv-g3-taus}, we arrive at \eqref{solv-g3}.

\smallskip

Next, we study the instanton condition for $\nreJF$.
First, we notice that the condition
$\treJF(1,2)=
-\frac{((1-2 \varepsilon)^2+4 \rho^2) (1+4 x^2) (4 x^2 s^4+t^4)}{16 x^2 r^2 (s^2 t^2-|v|^2)}=0$
directly implies that $(\re)=(\frac12,0)$.
Moreover, for these values, we get the following system of linear equations in $v_1$ and $v_2$:
\begin{equation*}
\begin{array}{lcl}
\Theta^{\frac12,0}_{\!(\!J,F)}(3,5)=0 & \Longleftrightarrow & (1+4x^2)t^2v_1+4x(s^2+t^2)v_2=0,\\[6pt]
\Theta^{\frac12,0}_{\!(\!J,F)}(3,6)=0 & \Longleftrightarrow &
-4x(s^2-t^2)v_1+(1+4x^2)t^2v_2=0.
\end{array}
\end{equation*}
Since the determinant is equal to $(1-4x^2)^2 t^4+16x^2s^4>0$, we have  $v_1=0=v_2$.
But in this case one gets
$$
\Upsilon^{\frac12,0}_{\!(\!J,F)}(1,3,1,4)=\frac{2x^2s^2(s^2+t^2)+t^4}{2xr^2s^2t^2}\not=0,
$$
so there are not instantons in the $(\re)$-plane.
\end{proof}

\begin{proposition}\label{solv-g5-balanced-real-basis}
Let $M$ be a solvmanifold with underlying Lie algebra isomorphic to $\frg_5$.
For every invariant balanced Hermitian structure $(J,F)$ on $M$, there is
a basis $\{e^k\}_{k=1}^6$ of 1-forms
satisfying \eqref{adapted-basis} and the equations
\begin{equation}\label{solv-g5}
\begin{cases}
\begin{array}{lcl}
de^1 \zzz & = &\zzz \frac{2}{t} e^{15},\\[4pt]
de^2 \zzz & = &\zzz \frac{2}{t} e^{25},\\[4pt]
de^3 \zzz & = &\zzz \frac{4 u_1}{t \sqrt{r^2 s^2-u_1^2}} e^{15}-\frac{2}{t} e^{35} ,\\[5pt]
 de^4 \zzz & = &\zzz \frac{4 u_1}{t \sqrt{r^2 s^2-u_1^2}} e^{25}-\frac{2}{t} e^{45} ,\\[6pt]
de^5 \zzz & = &\zzz 0 ,\\[5pt]
de^6 \zzz & = &\zzz \frac{2 t}{\sqrt{r^2 s^2-u_1^2}} e^{13} + \frac{2 t}{\sqrt{r^2 s^2-u_1^2}} e^{24},
\end{array}
\end{cases}
\end{equation}
where $r,s,t\in\mathbb{R}^*$ and $u_1\in \mathbb{R}$ satisfy $r^2 s^2>u_1^2$.

\smallskip

Furthermore, $\nreJF$ is never an instanton.
\end{proposition}

\begin{proof}
By \cite[Prop. 3.6]{FOU} there is only one complex structure $J$, up to isomorphism, with closed $(3,0)$-form on the Lie algebras $\frg_5$, and with complex structure equations
\begin{equation}\label{jotas-solv-g5}
d\omega^1\!=\omega^{1}\!\wedge (\omega^{3}+\omega^{\bar{3}}),\ \
d\omega^2\!=\! -\omega^{2}\!\wedge (\omega^{3}+\omega^{\bar{3}}),\ \
d\omega^3\!=\omega^{1\bar{2}}+\omega^{2\bar{1}}.
\end{equation}

An invariant Hermitian metric $F$ is balanced if and only if it is given by \eqref{2form} with $u=\bar{u}=u_1\in\R$ and $v=z=0$ \cite[Thm. 4.5]{FOU}.
In terms of the (1,0)-basis $\{\tau^{k}\}_{k=1}^3$
defined by
$$
\tau^{1}=\frac{\Delta}{s}\, \omega^{1}, \quad
\tau^{2}=\frac{u_1}{s}\, \omega^{1} +is\,\omega^{2}, \quad
\tau^{3}=t\,\omega^{3},
$$
where $\Delta=\sqrt{r^2s^2-u_1^2}$,
the metric can be written as $F=\frac{i}{2}(\tau^{1\bar1}+\tau^{2\bar2}+\tau^{3\bar3})$.
The complex structure equations \eqref{jotas-solv-g5} express in this basis as
\begin{equation}\label{jotas-solv-g5-taus}
\begin{cases}
\begin{array}{lcl}
d\tau^1 \zzz & = &\zzz \frac{1}{t} \tau^{13} +  \frac{1}{t} \tau^{1\bar{3}},\\[4pt]
d\tau^2 \zzz & = &\zzz \frac{2u_1}{t\,\Delta} \tau^{13} +  \frac{2u_1}{t\,\Delta} \tau^{1\bar{3}} - \frac{1}{t} \tau^{23} -  \frac{1}{t} \tau^{2\bar{3}},\\[4pt]
d\tau^3 \zzz & = &\zzz \frac{i\,t}{\Delta} \tau^{1\bar{2}} - \frac{i\,t}{\Delta}\tau^{2\bar{1}},
\end{array}
\end{cases}
\end{equation}
Taking the real basis $\{e^k\}_{k=1}^6$ as in Lemma~\ref{lema-solv}~(c),
from \eqref{jotas-solv-g5-taus} we arrive at the equations \eqref{solv-g5}.

\smallskip

Now, we study the connections $\nreJF$ that are instantons.
It can be checked that the condition $\ureJF(1,5,3,6)=0$ is satisfied if and only if $\rho=\varepsilon +\frac12$.
In this case, one has that $\Upsilon^{\varepsilon,\varepsilon+\frac12}_{\!(\!J,F)}(1,3,1,3)=0$ if and only if $\varepsilon=0$.
But taking $(\varepsilon,\rho)=(0,\frac12)$ we arrive at $\Theta^{0,\frac12}_{\!(\!J,F)}(5,6)=\frac{4 t^2}{r^2 s^2-u_1^2}$, which never vanishes. So, there are not instantons in the $(\re)$-plane.
\end{proof}

The following two propositions are devoted to the solvmanifolds with underlying Lie algebra isomorphic to $\frg_8$. When endowed with its complex parallelizable structure, it gives rise to the well-known Nakamura manifold \cite{Nakamura}. The following result studies this case, whereas in Proposition~\ref{solv-g8-A-balanced-real-basis} we focus on the complex structures (with closed $(3,0)$-form) which are not complex parallelizable.

\begin{proposition}\label{solv-g8-cp-balanced-real-basis}
Let $M$ be a solvmanifold with underlying Lie algebra isomorphic to $\frg_8$.
Let $J$ be its complex parallelizable structure and $F$ any invariant balanced Hermitian metric.
Then, there is a basis $\{e^k\}_{k=1}^6$ of 1-forms on $M$
satisfying \eqref{adapted-basis} and the equations
\begin{equation}\label{solv-g8-cp}
\begin{cases}
\begin{array}{lcl}
de^1 \zzz & = &\zzz -Y e^{16}-Y  e^{25},\\[4pt]
de^2 \zzz & = &\zzz Y e^{15}-Y  e^{26},\\[4pt]
de^3 \zzz & = &\zzz Z e^{15} -T e^{16}-T e^{25}-Z e^{26}+Y e^{36} + Y e^{45} ,\\[4pt]
de^4 \zzz & = &\zzz T e^{15} +Z e^{16}+Z e^{25}-T e^{26}-Y e^{35} + Y e^{46} ,\\[4pt]
de^5 \zzz & = &\zzz de^6=0,
\end{array}
\end{cases}
\end{equation}
where $Y,Z,T\in\mathbb{R}$, with $Y>0$.
Moreover, the connection $\nreJF$ is an instanton if and only if
it is the Chern connection (which is flat).
\end{proposition}

\begin{proof}
It is well-known that the complex parallelizable structure $J$ on $\frg_8$ is defined by the equations
\begin{equation}\label{jotas-solv-g8-cp}
d\omega^1= 2i\,\omega^{13},\ \
d\omega^2=-2i\,\omega^{23},\ \
d\omega^3=0,
\end{equation}
and any metric $F$ given by \eqref{2form} is balanced.
Let us consider the (1,0)-basis $\{\sigma^{k}\}_{k=1}^3$ defined by
$$
\sigma^{1}=\omega^{1}, \quad
\sigma^{2}= -\frac{iu}{s^2}\, \omega^{1}+\omega^{2} +\frac{i\bar{v}}{s^2}\,\omega^{3}, \quad
\sigma^{3}=\omega^{3}.
$$
Then, the complex structure equations \eqref{jotas-solv-g8-cp} express in
this basis as
\begin{equation}\label{jotas-solv-g8-cp-sigmas}
d\sigma^1 = 2i\,\sigma^{13},\quad
d\sigma^2 = \frac{4u}{s^2}\,\sigma^{13} -2i\,\sigma^{23},\quad
d\sigma^3 = 0,
\end{equation}
and the metric $F$ is written as
$$
F=\frac{i}{2}r'^2\sigma^{1\bar1}+\frac{i}{2}s'^2\sigma^{2\bar2}+\frac{i}{2}t'^2\sigma^{3\bar3}+\frac{z'}{2}\sigma^{1\bar3}-\frac{\bar{z'}}{2}\sigma^{3\bar1},
$$
where $r'^2=r^2-\frac{|u|^2}{s^2}$, $s'^2=s^2$, $t'^2=t^2-\frac{|v|^2}{s^2}$, and $z'=z+\frac{iuv}{s^2}$.

Let $\Delta'=\sqrt{r'^2t'^2-|z'|^2}$. Now, in terms of the new (1,0)-basis $\{\tau^{k}\}_{k=1}^3$ defined by
$$
\tau^{1}=r'\,\sigma^{1} + \frac{i \bar{z'}}{r'}\, \sigma^{3}, \quad
\tau^{2}=s'\, \sigma^{2}, \quad
\tau^{3}=\frac{\Delta'}{r'}\, \sigma^{3},
$$
the metric is $F=\frac{i}{2}(\tau^{1\bar1}+\tau^{2\bar2}+\tau^{3\bar3})$, and
the complex structure equations \eqref{jotas-solv-g8-cp-sigmas} express as
\begin{equation}\label{jotas-solv-g8-cp-taus}
\begin{cases}
\begin{array}{lcl}
d\tau^1 \zzz & = &\zzz \frac{2r'i}{\Delta'}\, \tau^{13},\\[4pt]
d\tau^2 \zzz & = &\zzz \frac{4u}{s'\Delta'}\, \tau^{13} -\frac{2r'i}{\Delta'}\, \tau^{23},\\[4pt]
d\tau^3 \zzz & = &\zzz 0.
\end{array}
\end{cases}
\end{equation}
Thus, taking the real basis $\{e^1,\ldots,e^6\}$ as in Lemma~\ref{lema-solv}~(c),
from \eqref{jotas-solv-g8-cp-taus} we arrive at \eqref{solv-g8-cp}
with $Y=\frac{2r'}{\Delta'}$ and $Z+iT=\frac{4u}{s'\Delta'}$. Notice that $Y\not=0$ and we can always suppose $Y>0$ (just by taking $-e^j$ for $j=5,6$).

\smallskip

We study next the instanton condition for the connections $\nreJF$.
Since $Y>0$, from the following two equations
\begin{equation*}
\begin{array}{lcl}
\ureJF(1,3,1,3)=0 & \Longleftrightarrow & ((1-2 \rho)^2 - 4 \varepsilon^2 )(2 Y^2+Z^2+T^2)=0,\\[6pt]
\treJF(5,6)=0 & \Longleftrightarrow &
\varepsilon\, (1-2 \rho) (2 Y^2+Z^2+T^2)=0,
\end{array}
\end{equation*}
it follows that $(\re)=(0,\frac12)$, i.e. we are reduced to study the Chern connection.
Now, by a direct calculation one can check that the instanton condition is satisfied for any metric, since all the curvature forms of the Chern connection vanish.

In conclusion, for any $J$-Hermitian metric $F$, the Chern connection is a flat instanton.
Note that
 $dF=2Y (e^{12}-e^{34})\wedge e^6 + (e^{13}+e^{24})\wedge(Te^5+Ze^6)
 - (e^{14}-e^{23})\wedge(Ze^5-Te^6)
 \not=0$.
\end{proof}

\begin{remark}\label{coefs-cp-g8}
The coefficients $Y,Z,T\in\mathbb{R}$ in the equations \eqref{solv-g8-cp} are related to the metric coefficients $r,s,t\in\mathbb{R}^*$ and $u,v,z\in \mathbb{C}$ in \eqref{2form}. Indeed, by the proof of Proposition~\ref{solv-g8-cp-balanced-real-basis} we have
$Y=Y(r,s,t,u,v,z)=\frac{2r'}{\Delta'}$, $Z=Z(r,s,t,u,v,z)=\frac{4u_1}{s'\Delta'}$ and $T=T(r,s,t,u,v,z)=\frac{4u_2}{s'\Delta'}$, where $r'^2=r^2-\frac{|u|^2}{s^2}$, $s'^2=s^2$, $t'^2=t^2-\frac{|v|^2}{s^2}$, $z'=z+\frac{iuv}{s^2}$, and $\Delta'=\sqrt{r'^2t'^2-|z'|^2}$.
\end{remark}

\begin{proposition}\label{solv-g8-A-balanced-real-basis}
Let $M$ be a solvmanifold with underlying Lie algebra isomorphic to $\frg_8$, endowed with an invariant $J$, which is not complex parallelizable, admitting balanced metrics.
Then, for every invariant balanced Hermitian structure $(J,F)$, there is
a basis $\{e^k\}_{k=1}^6$ of 1-forms on $M$
satisfying \eqref{adapted-basis} and the following equations:
\begin{equation}\label{solv-g8-A}
\begin{cases}
\begin{array}{lcl}
de^1 \zzz & = &\zzz -\frac{2a}{t} e^{15} - \frac{2}{t} e^{16}+\frac{2b}{t} e^{25},\\[4pt]
de^2 \zzz & = &\zzz -\frac{2b}{t} e^{15}-\frac{2a}{t} e^{25}-\frac{2}{t} e^{26},\\[4pt]
de^3 \zzz & = &\zzz -\frac{4(a\,u_1-b\,u_2)}{t \sqrt{1-|u|^2}}  e^{15} - \frac{4 u_1}{t \sqrt{1-|u|^2}} e^{16} + \frac{4(b\,u_1+a\,u_2)}{t \sqrt{1-|u|^2}}  e^{25} + \frac{4 u_2}{t \sqrt{1-|u|^2}} e^{26}\\[7pt]
\zzz &  &\zzz + \frac{2a}{t} e^{35} + \frac{2}{t} e^{36} - \frac{2b}{t} e^{45} ,\\[5pt]
 de^4 \zzz & = &\zzz -\frac{4(b\,u_1+a\,u_2)}{t \sqrt{1-|u|^2}}  e^{15} - \frac{4 u_2}{t \sqrt{1-|u|^2}} e^{16} - \frac{4(a\,u_1-b\,u_2)}{t \sqrt{1-|u|^2}}  e^{25} - \frac{4 u_1}{t \sqrt{1-|u|^2}} e^{26}\\[7pt]
\zzz &  &\zzz + \frac{2b}{t} e^{35} + \frac{2a}{t} e^{45} + \frac{2}{t} e^{46} ,\\[5pt]
de^5 \zzz & = &\zzz de^6=0,
\end{array}
\end{cases}
\end{equation}
where $(a,b)\in\mathbb{R}\times\mathbb{R}^*-\{(0,-1)\}$, and
$t\in \R^*$ and $u=u_1+i\,u_2\in \B=\{ u\in\C\mid |u|<1\}$.

\smallskip

Moreover, the connection $\nreJF$ is an instanton if and only if $(\varepsilon,\rho)=(0,\frac12)$ and $u=0$; that is, the (non-K\"ahler) balanced metric $F$ is ``diagonal'' and $\nreJF$ is the Chern connection (which is flat).
\end{proposition}

\begin{proof}
By \cite[Prop. 3.7 and Thm. 4.5]{FOU}, any complex structure $J$ with closed $(3,0)$-form on the Lie algebra $\frg_8$ admitting a balanced metric is given, up to isomorphism, by complex structure equations of the form \eqref{jotas-solv-KL} with $K=-A+i$ and $L=-A-i$,
where $A\in \mathbb{C}$ with $\Imag A\not=0$. The structure is complex parallelizable if and only if $A=-i$, so we suppose next that $A=a+b\,i$, with $b\not=0$ and $(a,b)\not=(0,-1)$.
Following Lemma~\ref{lema-solv}, we consider the basis of invariant 1-forms $\{e^k\}_{k=1}^6$ on $M$ defined in (c). Now, by a direct calculation from the complex equations given in Lemma~\ref{lema-solv}~(b) for $K=-a+(1-b)i$ and $L=-a-(1+b)i$ we arrive at \eqref{solv-g8-A}.

\smallskip

Now, the instanton conditions for the connection $\nreJF$
imply
\begin{equation*}
\begin{array}{lcl}
\ureJF(1,3,1,3)=0 & \Longleftrightarrow & ((1-2 \rho)^2 - 4 \varepsilon^2 )(1+a^2+b^2 |u|^2)=0,\\[6pt]
\treJF(5,6)=0 & \Longleftrightarrow &
\varepsilon\, (1-2 \rho) (1+a^2+b^2 |u|^2)=0.
\end{array}
\end{equation*}
Therefore, $(\re)=(0,\frac12)$ and we are reduced to study the Chern connection.
But a direct calculation shows that $\Theta^{0,\frac12}_{\!(\!J,F)}(1,2)=0$ if and only if $(a^2+(1+b)^2) |u|^2=0$, so $u=0$ (i.e. the balanced metric $F$ is diagonal when written as \eqref{2form-simplificada} in the basis $\{\omega^{k}\}_{k=1}^3$). Moreover, in this case all the curvature forms vanish identically.
Note that $dF=\frac{4 a}{t} e^{125} +\frac{4}{t} e^{126} - \frac{4 a}{t} e^{345} -\frac{4}{t} e^{346} \not=0$, i.e. the metric $F$ is not K\"ahler.
\end{proof}

\begin{remark}\label{conversely}
It is worthy to remark that the adapted equations given in   Sections~\ref{instantones-nil} and \ref{instantones-solv} not only provide a complete description of the spaces of invariant balanced Hermitian structures on their respective manifolds, they also give rise to a constructive converse result.
In other words, for any family of equations, one can apply the following constructive process, which we illustrate in the case of Proposition~\ref{solv-g1-balanced-real-basis}: choosing any $t\in \R^*$ and any $u=u_1+i\,u_2 \in \B=\{ u\in\C\mid |u|<1\}$, one can check that the equations \eqref{solv-g1} satisfy $d^2e^k=0$, $1\leq k\leq 6$, so they define a solvable Lie algebra which is
isomorphic to $\frg_1$; defining $J$ and $F$ by \eqref{adapted-basis} one has that  $J$ is a complex structure with non-zero closed $(3,0)$-form $\Psi=(e^1+i\,e^2)\wedge(e^3+i\,e^4)\wedge(e^5+i\,e^6)$,
and $F$ is a balanced $J$-Hermitian metric; finally, since
the corresponding simply-connected Lie group has a lattice \cite{FOU}, we get a compact solvmanifold endowed with the Hermitian structure $(J,F)$.
\end{remark}

\section{Solutions of the Hull-Strominger system and the heterotic equations of motion}\label{nuevas-soluciones}

\noindent In this section we find new explicit solutions of the Hull-Strominger system and the heterotic equations of motion on the compact solvmanifold underlying the
Nakamura manifold.
In the forthcoming paper \cite{OUV23} a general study of the system on solvmanifolds will be provided.

It follows from Sections~\ref{instantones-nil} and \ref{instantones-solv} that
the Chern connection $\nabla^c$ is an instanton in several cases,
although it is always flat. We have the following

\begin{corollary}\label{resumen-Chern}
Let $M=\Gamma\backslash G$ be a six-dimensional compact manifold defined as the quotient of a simply connected Lie group $G$ by a lattice $\Gamma$.
Suppose that $M$ possesses an invariant complex structure $J$ with non-zero closed $(3,0)$-form admitting balanced metrics~$F$.
Then,
the Chern connection $\nabla^c$ is an instanton in the following cases:

\medskip

\noindent \ \ $\bullet$ $\frg\cong\frh_5$, $\frg_8$ or $\mathfrak{so}(3,\!1)$, endowed with its complex parallelizable structure  $J$ (in this case any Hermitian metric is balanced);

\medskip

\noindent \ \ $\bullet$ $\frg\cong\frg_1$, $\frg_2^{\alpha} (\alpha\!>\!0)$ or $\frg_8$, $J$ is any complex structure and $F$ any ``diagonal'' Hermitian metric;

\medskip

\noindent \ \ $\bullet$ $\frg\cong\frg_2^0$, for any complex structure $J$ and any K\"ahler metric $F$.

\medskip

\noindent Moreover, in all the cases $\nabla^c$ is flat.
\end{corollary}

\begin{proof}
The result is a direct consequence of
Propositions~\ref{sl2C}, \ref{nilv-familia-2-14}, 
\ref{solv-g1-balanced-real-basis}, 
\ref{solv-g2-0-balanced-real-basis}, 
\ref{solv-g2-alfa-balanced-real-basis}, 
\ref{solv-g8-cp-balanced-real-basis} 
and~\ref{solv-g8-A-balanced-real-basis}. 
\end{proof}

Let $X$ be a compact complex manifold of complex dimension~$n$, endowed with a balanced Hermitian metric $F$. Then, $F^{n-1}$ defines a \emph{real} $(n-1,n-1)$ class in Bott-Chern cohomology.
We recall that the Bott-Chern cohomology groups \cite{bott-chern} of $X$ are defined by
$$
H^{p,q}_{BC}(X):=\frac{\ker d\colon \Omega^{p,q}(X,\C)\longrightarrow  \Omega^{p+q+1}(X,\C)}{{\rm im}\,d d^c \colon \Omega^{p-1,q-1}(X,\C)\longrightarrow  \Omega^{p,q}(X,\C)}  \;,
$$
and we will denote by $[F^{n-1}]$ the class defined by the balanced metric in $H^{n-1,n-1}_{BC}(X,\R)\subset H^{n-1,n-1}_{BC}(X)$.

Recall that $\frg_8$ is the Lie algebra of the solvmanifold underlying the Nakamura manifold. Next we consider it endowed with its abelian complex structure. In the following result we prove that for any given invariant balanced metric, there is another metric defining the same Bott-Chern (2,2)-class for which its associated Chern connection is flat and
the heterotic equations of motion are satisfied with respect to a non-flat instanton.

\begin{theorem}\label{new-solutions}
Let $X=(M,J_{\!Ab})$ be the Nakamura manifold endowed with its abelian complex structure $J_{\!Ab}$.
Let $F$ be any invariant balanced Hermitian metric on $X$ and $[F^2]\in H_{\rm BC}^{2,2}(X,\R)$  its corresponding Bott-Chern class. Then, there exists a balanced Hermitian metric $\tilde{F}$ on $X$ satisfying the following properties:
\begin{enumerate}
\item[{\rm (a)}] The metric $\tilde{F}$ is cohomologous to $F$, i.e. $[\tilde{F}^2]=[F^2]\in H_{\rm BC}^{2,2}(X,\R)$, and its associated Chern connection $\nabla_{(J_{\!Ab},\tilde{F})}^c$ is flat;
\item[{\rm (b)}] There is a connection $A$ compatible with $(J_{\!Ab},\tilde{F})$, which is a non-flat instanton and satisfies the heterotic equations of motion with respect to the Chern connection. 
\end{enumerate}
\end{theorem}

\begin{proof}
The abelian complex structure $J_{\!Ab}$ is given by the complex equations
\begin{equation}\label{ecus-g8-abeliana}
d\omega^1= -2i\,\omega^{1\bar{3}},\quad \
d\omega^2= 2i\,\omega^{2\bar{3}},\quad \
d\omega^3=0,
\end{equation}
so it corresponds to taking $K=0$ and $L=-2i$ in \eqref{jotas-solv-KL}.
By Lemma~\ref{lema-solv}~(a),
any invariant balanced Hermitian metric $F$ is given by \eqref{2form} with $v=z=0$.
In addition, we can normalize the metric coefficients $r=s=1$, so for the structure $J_{\!Ab}$ there is a (1,0)-basis $\{\omega^k\}_{k=1}^3$ satisfying \eqref{ecus-g8-abeliana} and such that any balanced metric $F$ is given by
\begin{equation}
2\,F=i\,(\omega^{1\bar1}+\omega^{2\bar2}+t^2\omega^{3\bar3})+u\,\omega^{1\bar2}-\bar{u}\,\omega^{2\bar1},
\end{equation}
with $t\in \R^*$ and $u\in \B=\{ u\in\C\mid |u|<1\}$. Then, we get
$$
2\,F^2=(1-|u|^2)\,\omega^{12\bar{1}\bar{2}} + t^2\omega^{13\bar{1}\bar{3}} + t^2\omega^{23\bar{2}\bar{3}} - iut^2\omega^{13\bar{2}\bar{3}} +
i \bar{u}t^2\omega^{23\bar{1}\bar{3}}.
$$
Notice that the equations \eqref{ecus-g8-abeliana} imply
$$
\partial\db (\omega^{1\bar{2}}) = 4\,\omega^{13\bar{2}\bar{3}},
\quad\quad
\partial\db (\omega^{2\bar{1}}) = 4\,\omega^{23\bar{1}\bar{3}},
$$
so we have
$$
F^2=\frac{1-|u|^2}{2}\,\omega^{12\bar{1}\bar{2}} + \frac{t^2}{2}\,\omega^{13\bar{1}\bar{3}} + \frac{t^2}{2}\,\omega^{23\bar{2}\bar{3}} + i\,\partial\db \left( \frac{ \bar{u}t^2}{8}\,\omega^{2\bar{1}} - \frac{ut^2}{8}\, \omega^{1\bar{2}} \right).
$$
Hence, $[F^2]=[\tilde{F}^2]$ in $H_{\rm BC}^{2,2}(X,\R)$, where
$\tilde{F}$ is the balanced Hermitian metric on $X$ defined by
\begin{equation}\label{Ftilde}
2\,\tilde{F}=i\,(\tilde{r}^2\,\omega^{1\bar1}+\tilde{s}^2\,\omega^{2\bar2}+\tilde{t}^{\,2}\,\omega^{3\bar3}),
\end{equation}
where $\tilde{r}=\tilde{s}=(1-|u|^2)^{1/4}$ and $\tilde{t}=t/(1-|u|^2)^{1/4}$.
Now, since the metric $\tilde{F}$ is diagonal, it follows from Corollary~\ref{resumen-Chern} that its Chern connection is flat. This proves (a).

\smallskip

To prove (b), we need to find an instanton $A$ solving the anomaly cancellation condition for the Chern connection, i.e. satisfying
\begin{equation}\label{anomaly-g8-abel}
dT=\frac{\alpha'}{4} ({\rm tr}\,\Omega^{c}\!\wedge\Omega^{c}-{\rm tr}\,\Omega^{A}\!\wedge\Omega^{A})=-\frac{\alpha'}{4} {\rm tr}\,\Omega^{A}\!\wedge\Omega^{A},
\end{equation}
where $\alpha'$ is a non-zero constant and $dT=dJ_{\!Ab}d\tilde{F}$, with $\tilde{F}$ given in \eqref{Ftilde}.
We first notice that again, because of the form of the complex equations \eqref{ecus-g8-abeliana}, there is an automorphism of $(\frg_8,J_{\!Ab})$ so that
we can suppose that $\tilde{r}=\tilde{s}=1$. From now on, we will denote the 
metric coefficient $\tilde{t}$ by $t$.

Hence,
we can write the balanced metric as $\tilde{F}=e^{12}+e^{34}+e^{56}$ with respect to an adapted basis $\{e^k\}_{k=1}^6$ whose differentials are given by~\eqref{solv-g8-A}
with $(a,b)=(0,1)$ and $u_1=u_2=0$. A direct calculation gives
$d\tilde{F}= \frac{4}{t} (e^{12} -  e^{34}) \wedge e^6$,
so we have the torsion 3-form $T=J_{\!Ab}d\tilde{F}=\frac{4}{t} (e^{12} -  e^{34}) \wedge e^5$. Thus,
\begin{equation}\label{dT-g8-abel}
dT= - \frac{16}{t^2} e^{1256} - \frac{16}{t^2} e^{3456}.
\end{equation}

Now, we consider a linear connection $A$ defined by the
connection 1-forms given, in the adapted basis $\{e^k\}_{k=1}^6$, by
$$
(\sigma^{A})^1_2=-(\sigma^{A})^2_1= \lambda\, e^1 + \mu\, e^2, \quad (\sigma^{A})^3_4=-(\sigma^{A})^4_3= \lambda\, e^3 + \mu\, e^4,
$$
and $(\sigma^{A})^i_j =0$ for any $(i,j)\not=(1,2),(2,1),(3,4),(4,3)$.
Here $\lambda,\mu$ are real numbers, so for any such pair one has a connection $A_{\lambda,\mu}$ (which we will denote simply by $A$).

Notice that any connection $A$ is Hermitian,
in fact, since we are working in an adapted basis, the compatibility of $A$ with the U(3)-structure $(J_{\!Ab},\tilde{F})$, i.e. $AJ_{\!Ab}=0=A\tilde{F}$,
is equivalent to the connection 1-forms to satisfy $(\sigma^{A})^j_i = -(\sigma^{A})^i_j$,
together with the conditions
$$
\begin{array}{l}
(\sigma^{A})^1_3 = (\sigma^{A})^2_4, \quad (\sigma^{A})^1_4 = -(\sigma^{A})^2_3, \quad
(\sigma^{A})^1_5 = (\sigma^{A})^2_6, \quad (\sigma^{A})^1_6 = -(\sigma^{A})^2_5,
\\[1em]
(\sigma^{A})^3_5 = (\sigma^{A})^4_6, \quad (\sigma^{A})^3_6 = -(\sigma^{A})^4_5.
\end{array}
$$
By a direct calculation we get that all the curvature 2-forms $(\Omega^{A})^i_j$ vanish, except for
$$
(\Omega^{A})^1_2= -(\Omega^{A})^2_1= - \frac{2 \mu}{t} e^{15} - \frac{2 \lambda}{t} e^{16} + \frac{2 \lambda}{t} e^{25} - \frac{2 \mu}{t} e^{26},
$$
and
$$
(\Omega^{A})^3_4= -(\Omega^{A})^4_3= \frac{2 \mu}{t} e^{35} + \frac{2 \lambda}{t} e^{36} - \frac{2 \lambda}{t} e^{45} + \frac{2 \mu}{t} e^{46}.
$$
Therefore, the conditions \eqref{Theta-Upsilon} in Lemma~\ref{condiciones-de-instanton} are fulfilled, which implies that $A$ is an instanton.
Moreover, from these curvature forms one obtains that the trace of the curvature of the instanton is
$$
{\rm tr}\,\Omega^{A}\!\wedge\Omega^{A}= - \frac{8}{t^2}(\lambda^2+\mu^2) e^{1256} - \frac{8}{t^2}(\lambda^2+\mu^2) e^{3456},
$$
so, taking into account \eqref{dT-g8-abel}, we have that the anomaly cancellation \eqref{anomaly-g8-abel} is satisfied if and only if
$$
- \frac{16}{t^2}= \frac{2\alpha'}{t^2}(\lambda^2+\mu^2).
$$
Finally, given any negative $\alpha'$, we can choose $\lambda,\mu$ such that
$\lambda^2+\mu^2=8/(-\alpha')$. Hence, with this choice, we have a solution of the heterotic equations of motion with non-flat instanton.
\end{proof}

\begin{remark}\label{het-ecus-A}
We note here that the Nakamura manifold endowed with its abelian complex structure $J_{\!Ab}$
also provides solutions to the heterotic equations of motion for any given positive $\alpha'$.
Indeed, we can exchange the roles of $\nabla^c$ and $A$ in the anomaly cancellation condition and look for solutions of the equation
$$
dT=\frac{\alpha'}{4} ({\rm tr}\,\Omega^{A}\!\wedge\Omega^{A}-{\rm tr}\,\Omega^{c}\!\wedge\Omega^{c})=\frac{\alpha'}{4} {\rm tr}\,\Omega^{A}\!\wedge\Omega^{A},
$$
with $\alpha'>0$. Notice that this lies precisely in the setting of the system proposed by M. Garcia-Fernandez in \cite{MGarcia-Crelle}, where metric connections which are instantons are allowed in the first term of the right hand side of the anomaly cancellation equation of the Hull-Strominger system. Recall that the instantons $A_{\lambda,\mu}$ found in Theorem~\ref{new-solutions} are not only metric, they are also Hermitian.
So, following the proof of the theorem, given any positive $\alpha'$, we can choose $\lambda,\mu$ such that $\lambda^2+\mu^2=8/\alpha'$. With this choice, one has many solutions of the heterotic equations of motion, according to \cite{MGarcia-Crelle}, with flat instanton.
\end{remark}

\vskip.2cm

\appendix

\section{}

\noindent In this appendix we
provide the relevant data for the proofs of the Propositions~\ref{prop_sl2R} and~\ref{resto_Lie_algebras}.
Tables~\ref{tabla_sl2R} and~\ref{tabla_su2} correspond to the $3\oplus 3$ decomposable Lie algebras $\mathfrak{sl}(2,\R)\oplus\mathfrak{g}_2$ and $\mathfrak{so}(3)\oplus\mathfrak{g}_2$, respectively. For the Lie algebras $L_{6,1}$ and $L_{6,4}$ the relevant data are given in Table~\ref{tabla_L61-L64}.


{\small\begin{table}[!hbt]
\centering
\renewcommand{\arraystretch}{1.4}
\begin{tabular}{>{$}l<{$}|>{$}l<{$}>{$}l<{$}}
\frg_2&\multicolumn{2}{c}{Lie algebra $\frg=\mathfrak{sl}(2,\R)\oplus\mathfrak{g}_2$  
}\\
\hline
\multirow{3}{*}{$\R^3$}&\multicolumn{2}{l}{$\rho= a_1 e^{123}+a_2 e^{124}+a_3 e^{125}+a_4 e^{126}+a_5 e^{134}+a_6 e^{135}+a_7 e^{136}+a_8 e^{234}+a_9e^{235}$}\\
&\quad\ \  +a_{10} e^{236}+a_{11} e^{456}&\\[1pt]
&\lambda=\mu=\tau=a_1a_{11},\quad A=B=C=0&\\
\hline
\multirow{4}{*}{$\frh_3$}&\multicolumn{2}{l}{$\rho=a_1 e^{123}+a_2 e^{124}+a_3 e^{125}+a_4 e^{134}+a_5 e^{135}+a_6 e^{234}+a_7 e^{235}+a_8 (e^{236}-e^{145})$} \\
&\quad\ \  +a_9 (e^{136}+e^{245})+a_{10} (e^{126}+e^{345})+a_{11} e^{456}&\\[1pt]
&\lambda=a_{10}^2 - a_1 a_{11} - a_8^2 - a_9^2,\quad \mu=a_{10}^2 - a_1 a_{11} + a_8^2 + a_9^2,\quad \tau=-a_{10}^2 - a_1 a_{11} + a_8^2 - a_9^2&\\
&A=-2 a_8 a_9,\quad B=-2 a_{10} a_8,\quad C=2 a_{10} a_9&\\
\hline
\multirow{6}{*}{$\fre(1,1)$}&\multicolumn{2}{l}{$\rho=a_1 e^{123}+a_2 e^{124}+a_3 e^{134}+a_4 e^{234}+a_5 (e^{146}+e^{235})+a_6 (e^{145}+e^{236})+a_7 (e^{245}-e^{136})$}\\
&\multicolumn{2}{l}{\quad\ \  $+a_8 (e^{246}-e^{135})+a_9 (e^{345}-e^{126})+a_{10} (e^{346}-e^{125})+a_{11} e^{456}$}\\[1pt]
&\lambda=a_{10}^2 - a_1 a_{11} - a_5^2 + a_6^2 + a_7^2 - a_8^2 - a_9^2,\quad \mu=a_{10}^2 - a_1 a_{11} + a_5^2 - a_6^2 - a_7^2 + a_8^2 - a_9^2 &\\
&\tau=-a_{10}^2 - a_1 a_{11} + a_5^2 - a_6^2 + a_7^2 - a_8^2 + a_9^2&\\
&A=2 (a_5 a_8- a_6 a_7),\quad B=2 (a_{10} a_5 - a_6 a_9),\quad
C=2 (a_{10} a_8 -  a_7 a_9)&\\
\hline
\multirow{6}{*}{$\fre(2)$}&\multicolumn{2}{l}{$\rho=a_1 e^{123}+a_2 e^{124}+a_3 e^{134}+a_4 e^{234}+a_5 (e^{146}+e^{235})+a_6 (e^{236}-e^{145})+a_7 (e^{136}+e^{245})$}\\
&\multicolumn{2}{l}{\quad\ \  $+a_8 (e^{246}-e^{135})+a_9 (e^{126}+e^{345})+a_{10} (e^{346}-e^{125})+a_{11} e^{456}$}\\[1pt]
&\lambda=a_{10}^2 - a_1 a_{11} - a_5^2 - a_6^2 - a_7^2 - a_8^2 + a_9^2,\quad \mu=a_{10}^2 - a_1 a_{11} + a_5^2 + a_6^2 + a_7^2 + a_8^2 + a_9^2 &\\
&\tau=-a_{10}^2 - a_1 a_{11} + a_5^2 + a_6^2 - a_7^2 - a_8^2 - a_9^2&\\
&A=2 (a_5 a_8 - a_6 a_7),\quad B=2 (a_{10} a_5 - a_6 a_9),\quad C=2 (a_{10} a_8 + a_7 a_9)&\\
\hline
\multirow{6}{*}{$\mathfrak{sl}(2,\R)$}&\multicolumn{2}{l}{$\rho=a_1 e^{123}+a_2 (e^{234}\!-\!e^{156})+a_3 (e^{146}+e^{235})+a_4 (e^{145}+e^{236})+a_5 (e^{245}\!-\!e^{136})+a_6 (e^{246}\!-\!e^{135})$}\\
&\multicolumn{2}{l}{\quad\ \  $+a_7 (e^{134}+e^{256})+a_8 (e^{345}-e^{126})+a_9 (e^{346}-e^{125})+a_{10} (e^{124}+e^{356})+a_{11} e^{456}$}\\[1pt]
&\lambda=-a_{10}^2 - a_1 a_{11} + a_2^2 - a_3^2 + a_4^2 + a_5^2 - a_6^2 + a_7^2 - a_8^2 +
a_9^2&\\
&\mu=-a_{10}^2 - a_1 a_{11} - a_2^2 + a_3^2 - a_4^2 - a_5^2 + a_6^2 - a_7^2 - a_8^2 +
a_9^2 &\\
&\tau=a_{10}^2 - a_1 a_{11} - a_2^2 + a_3^2 - a_4^2 + a_5^2 - a_6^2 + a_7^2 + a_8^2 - a_9^2&\\
&A=-2 (a_4 a_5 - a_3 a_6 + a_2 a_7),\quad B=-2 (a_{10} a_2 + a_4 a_8 - a_3 a_9),\quad C=-2 (a_{10} a_7 + a_5 a_8 - a_6 a_9)&\\
\hline
\multirow{6}{*}{$\mathfrak{so}(3)$}&\multicolumn{2}{l}{$\rho=a_1 e^{123}+a_2 (e^{234}\!-\!e^{156})+a_3 (e^{146}+e^{235})+a_4 (e^{236}\!-\!e^{145})+a_5 (e^{136}+e^{245})+a_6 (e^{246}\!-\!e^{135})$}\\
&\multicolumn{2}{l}{\quad\ \  $+a_7 (e^{134}+e^{256})+a_8 (e^{126}+e^{345})+a_9 (e^{346}-e^{125})+a_{10} (e^{124}+e^{356})+a_{11} e^{456}$}\\[1pt]
&\lambda=a_{10}^2 - a_1 a_{11} - a_2^2 - a_3^2 - a_4^2 - a_5^2 - a_6^2 - a_7^2 + a_8^2 + a_9^2 &\\
&\mu=a_{10}^2 - a_1 a_{11} + a_2^2 + a_3^2 + a_4^2 + a_5^2 + a_6^2 + a_7^2 + a_8^2 + a_9^2&\\
&\tau=-a_{10}^2 - a_1 a_{11} + a_2^2 + a_3^2 + a_4^2 - a_5^2 - a_6^2 - a_7^2 - a_8^2 - a_9^2&\\
&A=-2 (a_4 a_5 - a_3 a_6 + a_2 a_7),\quad B=-2 (a_{10} a_2 + a_4 a_8 - a_3 a_9),\quad C=2 (a_{10} a_7 + a_5 a_8 + a_6 a_9)&
\end{tabular}

\vskip.5cm

\caption{Closed 3-forms $\rho$ on $\frg=\mathfrak{sl}(2,\R)\oplus\mathfrak{g}_2$ and the values $\lambda,\mu,\tau,A,B,C$ (up to the constant $|\tilde \lambda(\rho)|^{-1/2}$) of the linear endomorphism $F$ of $\mathfrak{sl}(2,\R)$ in Lemma~\ref{lema_sl2R}~(i) induced by $\tilde J_\rho$.}\label{tabla_sl2R}

\end{table}}


\newpage


{\small\begin{table}[!hbt]
\centering
\renewcommand{\arraystretch}{1.4}
\begin{tabular}{>{$}l<{$}|>{$}l<{$}>{$}l<{$}}
\frg_2&\multicolumn{2}{c}{Lie algebra $\frg=\mathfrak{so}(3)\oplus\mathfrak{g}_2$
}\\
\hline
\multirow{3}{*}{$\R^3$}&\multicolumn{2}{l}{$\rho=a_1 e^{123}+a_2 e^{124}+a_3 e^{125}+a_4 e^{126}+a_5 e^{134}+a_6 e^{135}+a_7 e^{136}+a_8 e^{234}+a_9 e^{235}$}\\
&\quad\ \  +a_{10} e^{236}+a_{11} e^{456}\\[1pt]
&\lambda=\mu=\tau=-a_1a_{11},\quad A=B=C=0&\\
\hline
\multirow{4}{*}{$\frh_3$}&\multicolumn{2}{l}{$\rho=a_1 e^{123} + a_2 e^{124}+a_3 e^{125}+a_4 e^{126}+a_5 e^{134}+a_6 e^{135}+a_7 e^{136}+a_8 e^{234}+a_9 e^{235}$}\\
&\quad\ \  +a_{10} e^{236}+a_{11} e^{456}\\[1pt]
&\lambda=-a_{10}^2 - a_1 a_{11} + a_8^2 - a_9^2,\quad \mu=-a_{10}^2 - a_1 a_{11} - a_8^2 + a_9^2,\quad \tau=a_{10}^2 - a_1 a_{11} - a_8^2 - a_9^2\\
&A=-2 a_8 a_9,\quad B=-2 a_{10} a_8,\quad C=2 a_{10} a_9&\\
\hline
\multirow{6}{*}{$\fre(1,1)$}&\multicolumn{2}{l}{$\rho=a_1 e^{123}+a_2 e^{124}+a_3 e^{134}+a_4 e^{234}+a_5 (e^{146}+e^{235})+a_6 (e^{145}+e^{236})+a_7 (e^{245}-e^{136}) $}\\
&\multicolumn{2}{l}{\quad\ \  $+a_8 (e^{246}-e^{135})+a_9 (e^{126}+e^{345})+a_{10} (e^{125}+e^{346})+a_{11} e^{456}$}\\[1pt]
&\lambda=-a_{10}^2 - a_1 a_{11} + a_5^2 - a_6^2 + a_7^2 - a_8^2 + a_9^2,\quad \mu=-a_{10}^2 - a_1 a_{11} - a_5^2 + a_6^2 - a_7^2 + a_8^2 + a_9^2&\\
&\tau=a_{10}^2 - a_1 a_{11} - a_5^2 + a_6^2 + a_7^2 - a_8^2 - a_9^2&\\
&A=2 (a_5 a_8 - a_6 a_7),\quad B=2 (a_{10} a_5 - a_6 a_9),\quad
C=2 (a_{10} a_8 - a_7 a_9)&\\
\hline
\multirow{6}{*}{$\fre(2)$}&\multicolumn{2}{l}{$\rho=a_1 e^{123}+a_2 e^{124}+a_3 e^{134}+a_4 e^{234}+a_5 (e^{146}+e^{235})+a_6 (e^{236}-e^{145})+a_7 (e^{136}+e^{245})$}\\
&\multicolumn{2}{l}{\quad\ \  $+a_8 (e^{246}-e^{135})+a_9 (e^{345}-e^{126})+a_{10} (e^{125}+e^{346})+a_{11} e^{456}$}\\[1pt]
&\lambda=-a_{10}^2 - a_1 a_{11} + a_5^2 + a_6^2 - a_7^2 - a_8^2 - a_9^2,\quad \mu=-a_{10}^2 - a_1 a_{11} - a_5^2 - a_6^2 + a_7^2 + a_8^2 - a_9^2&\\
&\tau=a_{10}^2 - a_1 a_{11} - a_5^2 - a_6^2 - a_7^2 - a_8^2 + a_9^2&\\
&A=2 (a_5 a_8 - a_6 a_7),\quad B=2 (a_{10} a_5 - a_6 a_9),\quad C=2 (a_{10} a_8 + a_7 a_9)&\\
\hline
\multirow{6}{*}{$\mathfrak{so}(3)$}&\multicolumn{2}{l}{$\rho=a_1 e^{123}+a_2 (e^{234}-e^{156})+a_3 (e^{146}+e^{235})+a_4 (e^{236}-e^{145})+a_5 (e^{136}+e^{245})+a_6 (e^{246}-e^{135})$}\\
&\multicolumn{2}{l}{\quad\ \  $+a_7 (e^{134}+e^{256})+a_8 (e^{126}+e^{345})+a_9 (e^{346}-e^{125})+a_{10} (e^{124}+e^{356})+a_{11} e^{456}$}\\[1pt]
&\lambda=-a_{10}^2 - a_1 a_{11} + a_2^2 + a_3^2 + a_4^2 - a_5^2 - a_6^2 - a_7^2 - a_8^2 -
a_9^2&\\
&\mu=-a_{10}^2 - a_1 a_{11} - a_2^2 - a_3^2 - a_4^2 + a_5^2 + a_6^2 + a_7^2 - a_8^2 -
a_9^2&\\
&\tau=a_{10}^2 - a_1 a_{11} - a_2^2 - a_3^2 - a_4^2 - a_5^2 - a_6^2 - a_7^2 + a_8^2 + a_9^2&\\
&A=-2 (a_4 a_5 - a_3 a_6 + a_2 a_7),\quad B=-2 (a_{10} a_2 + a_4 a_8 - a_3 a_9),\quad
C=2 (a_{10} a_7 + a_5 a_8 + a_6 a_9)&
\end{tabular}

\vskip.5cm

\caption{Closed 3-forms $\rho$ on $\frg=\mathfrak{so}(3)\oplus\mathfrak{g}_2$ and the values $\lambda,\mu,\tau,A,B,C$ (up to the constant $|\tilde \lambda(\rho)|^{-1/2}$) of the linear endomorphism $F$ of $\mathfrak{so}(3)$ in Lemma~\ref{lema_sl2R}~(ii) induced by $\tilde J_\rho$.}\label{tabla_su2}

\end{table}
}


\newpage

\begin{table}[!hbt]
  \centering

  \renewcommand{\arraystretch}{1.4}
  \begin{tabular}{>{$}c<{$}|>{$}l<{$}}
    \text{$\frg$}&\multicolumn{1}{c}{}\\
    \hline
    &\\[-15pt]
  \multirow{17}{*}{$L_{6,1}$}&\rho=a_1 e^{123} + a_2 e^{126} + a_3 e^{135} + a_4 (e^{136}-e^{125}) + a_5 e^{234} + a_6 (e^{235}-e^{134}) \\[-2pt]
      &\quad\ \ +\, a_7 (e^{124} + e^{236})  + a_8 (e^{146}+e^{256}) + a_9 (e^{245} + e^{346}) +a_{10}(e^{356}-e^{145})  + a_{11} e^{456},\\
      &d\big(\tilde J_\rho\rho\big)=q_1e^{1234}+q_2e^{1235}+q_3e^{1236}+q_4e^{1245}+q_5(e^{1246}+e^{1345})+q_6(e^{1256}+e^{2345})\\[-2pt]
      &\quad\quad\quad\quad\  +\,q_7e^{1346}+q_8(e^{1356}+e^{2346})+q_9e^{2356},\\
      &\tilde \lambda(\rho)=(a_1a_{11})^2+a_{10}q_2+a_3q_4+a_4q_5-a_6q_8+a_3q_9,\quad \textrm{where:}\\
      &q_1=4(a_{10}a_2a_6+a_{10}a_4a_7-a_4a_6a_8+a_3a_7a_8+a_2a_3a_9+a_4^2a_9),\\
      &q_2=-4(a_{10}a_2a_5-a_{10}a_7^2-a_4a_5a_8-a_6a_7a_8-a_2a_6a_9-a_4a_7a_9),\\
      &q_3=-4(a_{10}a_4a_5+a_{10}a_6a_7+a_3a_5a_8+a_6^2a_8-a_4a_6a_9+a_3a_7a_9),\\
      &q_4=2(a_{10}^2a_2+a_{11}a_2a_3+a_{11}a_4^2-a_{10}^2a_5-a_{11}a_2a_5-a_{11}a_3a_5-a_{11}a_6^2+a_{11}a_7^2\\
      &\quad\quad\quad -2a_{10}a_4a_8  -a_3a_8^2+a_5a_8^2+2a_{10}a_6a_9+2a_7a_8a_9+a_2a_9^2+a_3a_9^2),\\
      &q_5=-4(a_{11}a_4a_5+a_{11}a_6a_7-a_{10}a_5a_8-a_{10}a_7a_9+a_6a_8a_9-a_4a_9^2),\\
      &q_6=-4(a_{11}a_4a_6-a_{10}^2a_7-a_{11}a_3a_7-a_{10}a_6a_8-a_{10}a_4a_9-a_3a_8a_9),\\
      &q_7=2(a_{10}^2a_2+a_{11}a_2a_3+a_{11}a_4^2+a_{10}^2a_5+a_{11}a_2a_5+a_{11}a_3a_5+a_{11}a_6^2-a_{11}a_7^2\\
      &\quad\quad\quad -2a_{10}a_4a_8 -a_3a_8^2-a_5a_8^2-2a_{10}a_6a_9-2a_7a_8a_9-a_2a_9^2-a_3a_9^2),\\
      &q_8=4(a_{11}a_2a_6+a_{11}a_4a_7-a_{10}a_7a_8-a_6a_8^2-a_{10}a_2a_9+a_4a_8a_9),\\
      &q_9=-2(a_{10}^2a_2+a_{11}a_2a_3+a_{11}a_4^2-a_{10}^2a_5+a_{11}a_2a_5-a_{11}a_3a_5-a_{11}a_6^2-a_{11}a_7^2\\
      &\quad\quad\quad\quad -2a_{10}a_4a_8 -a_3a_8^2-a_5a_8^2+2a_{10}a_6a_9-2a_7a_8a_9-a_2a_9^2+a_3a_9^2).\\
\hline
&\\[-15pt]

\multirow{17}{*}{$L_{6,4}$}&\rho= a_1 e^{123} + a_2 e^{124} + a_3 (e^{126} + e^{134}) + a_4 e^{136} + a_5(e^{234}-2e^{125})  + a_6 e^{235}  \\[-2pt]
  &\quad\ \ +\, a_7 (e^{236}-2e^{135})  + a_8 (2e^{145}+e^{246}) + a_9 (e^{256} + e^{345}) + a_{10}(e^{346}-2e^{156})  + a_{11} e^{456},\\
  &d\big(\tilde J_\rho\rho\big)=q_1e^{1234}+q_2e^{1235}+q_3e^{1236}+q_4e^{1245}+q_5(e^{1246}+e^{2345})+q_6(e^{1256}+e^{1345})\\[-2pt]
  &\quad\quad\quad\quad\ +\,q_7(e^{1346}-e^{2356})+q_8e^{1356}+q_9e^{2346},\\
  &\tilde \lambda(\rho)=(a_1a_{11})^2+a_{10}q_1-\frac{a_4}{2}q_4-a_7q_5+\frac{a_3}{2}q_6+a_3q_9,\quad \textrm{where:}\\
  &q_1= 8(2a_{10}a_5^2+a_{10}a_2a_6-a_3a_6a_8-2a_5a_7a_8+a_3a_5a_9-a_2a_7a_9),\\
  &q_2= -8(2a_{10}a_3a_5-2a_{10}a_2a_7-2a_4a_5a_8+2a_3a_7a_8+a_3^2a_9-a_2a_4a_9),\\
  &q_3= -8(a_{10}a_3a_6+2a_{10}a_5a_7-a_4a_6a_8-2a_7^2a_8+a_4a_5a_9-a_3a_7a_9),\\
  &q_4= 8(2a_{11}a_5^2+a_{11}a_2a_6-2a_6a_8^2+4a_5a_8a_9+a_2a_9^2),\\
  &q_5= -8(a_{11}a_3a_5-a_{11}a_2a_7-2a_{10}a_5a_8+2a_7a_8^2-a_{10}a_2a_9+a_3a_8a_9),\\
  &q_6= 4(2a_{10}^2a_2+a_{11}a_3^2-a_{11}a_2a_4-4a_{10}a_3a_8+2a_4a_8^2),\\
  &q_7= -8(2a_{10}^2a_5-a_{11}a_4a_5+a_{11}a_3a_7-2a_{10}a_7a_8+a_{10}a_3a_9-a_4a_8a_9),\\
  &q_8= -8(2a_{10}^2a_6-a_{11}a_4a_6-2a_{11}a_7^2-4a_{10}a_7a_9-a_4a_9^2),\\
  &q_9= -2(2a_{10}^2a_2+a_{11}a_3^2-a_{11}a_2a_4-2a_{11}a_3a_6-4a_{11}a_5a_7-4a_{10}a_3a_8+4a_{10}a_6a_8\\[-2pt]
  &\quad\quad\quad\quad  +\,2a_4a_8^2-4a_{10}a_5a_9-4a_7a_8a_9-2a_3a_9^2).
  \end{tabular}

  \vskip.5cm

  \caption{Generic closed three-forms $\rho$ and the corresponding four-form $d(\tilde J_\rho\rho)$ and scalar $\tilde \lambda(\rho)$ on the Lie algebra $\frg=L_{6,1}, L_{6,4}$.}\label{tabla_L61-L64}

\end{table}

\begin{landscape}

\end{landscape}

\section*{Acknowledgments}
\noindent
This work has been partially supported by grant PID2020-115652GB-I00, funded by MCIN/AEI/10.13039/501100011033,
and by grant E22-17R ``Algebra y Geometr\'{\i}a'' (Gobierno de Arag\'on/FEDER).

\end{document}